\documentclass[10pt,a4paper]{amsart}
\usepackage{amssymb}
\usepackage{amsmath}
\usepackage{xspace}
\usepackage[english]{babel}
\usepackage[latin1]{inputenc}
\usepackage[all,2cell]{xy}

\newdir{<= }{:a(+25)\dir{-}*:a(-25)\dir{-}*!/:a(85)+1.8pt/:a(-25)\dir{-}}
\newdir{ = }{*@{=}}
\newdir{+>}{@{|}*@{-}!/-6.5pt/@{>}*!/-13pt/{}}
\newdir{ +}{{}*!/-5pt/@{+}}
\newdir{ >}{{}*!/-10pt/@{>}}
\newdir{.-}{{}*:a(-90)h!/2.1pt/{.}}
\newdir{.->}{:a(-90)h!/2.5pt/{.}*!/4pt/\dir{-}*!/0pt/\dir{-}*%
!/-2pt/\dir{-}*!/-8pt/\dir{>}}
\newdir{-.->}%
{:a(-90)h!/2pt/{.}*!/8pt/\dir{-}*!/4pt/\dir{-}*!/0pt/\dir{-}*%
!/-5pt/\dir{-}*!/-11pt/\dir{>}}
\newdir{~>}{!/4.6pt/\dir{~}*!/1pt/\dir{-}*!/-5pt/\dir{>}}
\newdir{~~>}{!/7pt/\dir{~}*!/0pt/\dir{~}*!/-5pt/\dir{>}}
\newdir{=~>}{*!/1pt/@2{~}*!/6pt/{}*!/-6pt/@2{>}}
\newdir{<~~>}{!/21pt/\dir{<}*!/21pt/\dir{-}!/12pt/\dir{~}!/5pt/\dir{~}*!/1pt/\dir{-}*!/-3pt/\dir{>}}
\newdir{=>}{!/5pt/\dir{=}!/2.5pt/\dir{=}*!/-5pt/\dir2{>}*!/7pt/\dir{ }}
\newdir{==>}{!/-2pt/\dir{=}!/-6pt/\dir{=}%
 !/-10pt/\dir{=}*!/-18pt/\dir2{>}}
\newdir{<==>}{!\dir2{<}!/-2pt/\dir{=}!/-6pt/\dir{=}%
 !/-10pt/\dir{=}*!/-17pt/\dir2{>}}
\newdir{< }{{}*!/12pt/\dir{<}}
\newdir{-o}{!/1.3pt/{}*!/0pt/{}@{o}*!/-5pt/{}}

\newsavebox{\xymor}  
\newsavebox{\xymon}  
\newsavebox{\xyepi}  
\newsavebox{\xytn}   
\newsavebox{\xyrel}  
\newsavebox{\xycel}  
\newsavebox{\xymdf}  
\newsavebox{\xyumor} 
\newsavebox{\xydmor} 
\newsavebox{\xyomor} 
\newsavebox{\xyemor} 

\newcommand{\xynode}{\makebox[0ex]{}}
\savebox{\xymor}{\ensuremath{%
\xymatrix@1@C=19pt{\xynode \ar@{>}[r] & \xynode }}}
\savebox{\xymon}{\ensuremath{%
\xymatrix@1@C=19pt{\xynode \ar@{{ +}{-}{>}}[r] & \xynode }}}
\savebox{\xyepi}{\ensuremath{%
\xymatrix@1@C=19pt{\xynode \ar@{{}{-}{+>}}[r] & \xynode }}}
\savebox{\xytn}{\ensuremath{%
\xymatrix@1@C=19pt{\xynode \ar[r]|(.44){\object@{.-}} & \xynode
}}}
\savebox{\xyrel}{\ensuremath{%
\xymatrix@1@C=19pt{\xynode \ar@{{}{-}{-o}}[r] & \xynode }}}
\savebox{\xycel}{\ensuremath{%
\xymatrix@1@C=19pt{\xynode \ar@{=>}[r] & \xynode }}}
\savebox{\xymdf}{\ensuremath{%
\xymatrix@1@C=16pt{\xynode \ar@{}[r]|{\dir{~>}} & \xynode}}}
\savebox{\xyumor}{\ensuremath{%
\xymatrix@1@C=19pt{\xynode \ar@{{}{-}^{>}}[r] & \xynode }}}
\savebox{\xydmor}{\ensuremath{%
\xymatrix@1@C=19pt{\xynode \ar@{{}{-}_{>}}[r] & \xynode }}}
\savebox{\xyomor}{\ensuremath{%
\xymatrix@1@C=19pt{\xynode \ar@{{}{-}^{< }}[r] & \xynode }}}
\savebox{\xyemor}{\ensuremath{%
\xymatrix@1@C=19pt{\xynode \ar@{{ >}{-}{>}}[r] & \xynode }}}

\newcommand{\mor}{\usebox{\xymor}}    
\newcommand{\cel}{\usebox{\xycel}}    

\UseAllTwocells
\UseAllTwocells

\theoremstyle{plain}
\newtheorem{theorem}{Theorem}[section]
\newtheorem{proposition}[theorem]{Proposition}
\newtheorem{corollary}[theorem]{Corollary}
\newtheorem{lemma}[theorem]{Lemma}

\theoremstyle{definition}
\newtheorem{definition}[theorem]{Definition}

\newtheorem*{remark}{Remark}
\newtheorem*{assumption}{Assumption}
\newtheorem*{conventions}{Conventions}
\theoremstyle{remark}

\numberwithin{equation}{section}

\begin{document}
\title[Profinite many-sorted algebras and ultraproducts]
{When are profinite many-sorted algebras retracts of ultraproducts of finite many-sorted algebras?
}

\author[Climent]{J. Climent Vidal}
\address{Universitat de Val\`{e}ncia\\
         Departament de L\`{o}gica i Filosofia de la Ci\`{e}ncia\\
         Av. Blasco Ib\'{a}\~{n}ez, 30-$7^{\mathrm{a}}$, 46010 Val\`{e}ncia, Spain}
\email{Juan.B.Climent@uv.es}
\author[Cosme]{E. Cosme Ll\'{o}pez}
\address{Universitat de Val\`{e}ncia\\
         Departament d'\`{A}lgebra\\
         Dr. Moliner, 50, 46100 Burjassot, Val\`{e}ncia, Spain}
\email{Enric.Cosme@uv.es}

\keywords{Support of a many-sorted set, family of many-sorted algebras with constant support, profinite, retract, projective limit, inductive limit, ultraproduct.}
\date{\today}

\begin{abstract}
For a set of sorts $S$ and an $S$-sorted signature $\Sigma$ we prove that a profinite $\Sigma$-algebra, i.e., a projective limit of a projective system of finite $\Sigma$-algebras, is a retract of an ultraproduct of finite $\Sigma$-algebras if the family consisting of the finite $\Sigma$-algebras underlying the projective system is with constant support. In addition, we provide a categorial rendering of the above result. Specifically, after obtaining a category where the objects are the pairs formed by a nonempty upward directed preordered set and by an ultrafilter containing the filter of the final sections of it, we show that there exists a functor from the just mentioned category whose object mapping assigns to an object a natural transformation which is a retraction.
\end{abstract}
\maketitle


\section{Introduction.}

In their article ``Profinite structures are retracts of ultraproducts of finite structures''~\cite{mm07}, H. L. Mariano and F. Miraglia proved, for a single-sorted first order language with equality  $\mathcal{L}$, that the profinite $\mathcal{L}$-algebraic systems, i.e., the projective limits of finite $\mathcal{L}$-algebraic systems, are retracts of certain ultraproducts of finite $\mathcal{L}$-algebraic systems.

It is true that, broadly speaking, almost all fundamental statements from single-sorted algebra (or single-sorted equational logic), when suitably translated, are also valid for many-sorted algebra (or many-sorted equational logic). However, there are statements from single-sorted algebra which can not be generalized to many-sorted algebras without some type of qualification, which is ultimately grounded on the fact that many-sorted equational logic is not an inessential variation of single-sorted equational logic. (Some examples of theorems about single-sorted algebras which do not go through in their original form to the setting of many-sorted algebras can be found e.g., in~\cite{cs04}--\cite{cs16}, \cite{gm85}, \cite{m76}, and \cite{mrs13}.) In this connection, the aforementioned result of Mariano and Miraglia is no exception and in order to be adapted to many-sorted algebras, it will also require some adjustment. Accordingly, for an arbitrary set of sorts $S$ and an arbitrary $S$-sorted signature $\Sigma$, the main objective of this article is to establish a sufficient (and natural) condition for a profinite $\Sigma$-algebra to be a retract of an ultraproduct of finite $\Sigma$-algebras (let us notice that after having done that, the extension of this result to the case of a many-sorted first order language with equality $\mathcal{L}$ and  $\mathcal{L}$-algebraic systems is straightforward). We point out that the required adjustment is, ultimately, founded on the concept of support mapping for the set of sorts $S$ and on the notion of family of $\Sigma$-algebras with constant support (details will be found in the penultimate section of this article).


We next proceed to succinctly summarize the contents of the subsequent sections of this article. The reader will find a more detailed explanation at the beginning of the succeeding sections.

In Section 2, for the convenience of the reader, we recall, mostly without proofs, for a set of sorts $S$ and an $S$-sorted signature $\Sigma$, those notions and constructions of the theories of $S$-sorted sets and of $\Sigma$-algebras which are indispensable to define in the following section those others which will allow us to achieve the above  mentioned  main results, thus making, so we hope, our exposition self-contained.

After having stated all of these auxiliary results we provide in Section 3 a solution to the problem posed in the title of this article. Concretely, we prove, for an $S$-sorted signature $\Sigma$, the following proposition:
{
\begin{quotation}
\noindent If $\mathbf{A}$ is a profinite $\Sigma$-algebra, i.e., a projective limit of a projective system $\boldsymbol{\mathcal{A}}$ of finite $\Sigma$-algebras relative to a nonempty upward directed preordered set $\mathbf{I} = (I,\leq)$, and $(\mathbf{A}^{i})_{i\in I}$, the underlying family of finite $\Sigma$-algebras of $\boldsymbol{\mathcal{A}}$, is with constant support, then, for a suitable ultrafilter $\mathcal{F}$ on $I$, we have that $\mathbf{A}$ is a retract of $\prod_{i\in I}\mathbf{A}^{i}/\equiv^{\mathcal{F}}$, the ultraproduct of $(\mathbf{A}^{i})_{i\in I}$ relative to $\mathcal{F}$.
\end{quotation}
}

Finally, in Section 4, after obtaining, by means of the Grothendieck construction for a covariant functor from a convenient  category of nonempty upward directed preordered sets to the category of sets, a category in which the objects are the pairs formed by a nonempty upward directed preordered set and by an ultrafilter containing the filter of the final sections of it, we provide a categorial rendering of the aforementioned many-sorted version of Mariano-Miraglia theorem. Specifically, we show that there exists a functor from the just mentioned category whose object mapping assigns to an object a natural transformation, between two functors from a suitable category of projective systems of $\Sigma$-algebras to the category of $\Sigma$-algebras, which is a retraction.

Our underlying set theory is $\mathbf{ZFSk}$, Zermelo-Fraenkel-Skolem set theory (also known as $\mathbf{ZFC}$, i.e.,  Zermelo-Fraenkel set theory with the axiom of choice) plus the existence of a Grothendieck universe $\ensuremath{\boldsymbol{\mathcal{U}}}$, fixed once and for all (see~\cite{sM98}, pp.~21--24). We recall that the elements of $\ensuremath{\boldsymbol{\mathcal{U}}}$ are called $\ensuremath{\boldsymbol{\mathcal{U}}}$-small sets and the subsets of $\ensuremath{\boldsymbol{\mathcal{U}}}$ are called $\ensuremath{\boldsymbol{\mathcal{U}}}$-large sets or classes. Moreover, from now on $\mathbf{Set}$ stands for the category of sets, i.e., the category whose set of objects is $\ensuremath{\boldsymbol{\mathcal{U}}}$ and whose set of morphisms is $\bigcup_{A,B\in \boldsymbol{\mathcal{U}}}\mathrm{Hom}(A,B)$, the set of all mappings between $\ensuremath{\boldsymbol{\mathcal{U}}}$-small sets.

In all that follows we use standard concepts and constructions from category theory, see~\cite{hs73}, \cite{sM98}, and \cite{em76}, and from many-sorted algebra, see~\cite{m76} and \cite{w92}. More specific notational and conceptual conventions will be included and explained in the following section.

\section{Preliminaries.}


In this section we introduce those basic notions and constructions which we shall need to obtain the aforementioned main result of this article. Specifically, for a set (of sorts) $S$ in $\ensuremath{\boldsymbol{\mathcal{U}}}$, we begin by recalling the concept of free monoid on $S$, which will be fundamental for defining the concept of $S$-sorted signature. Following this we define the concepts of $S$-sorted set, $S$-sorted mapping from an $S$-sorted set to another, and the corresponding category. Moreover, we define the subset relation between $S$-sorted sets, the notion of finiteness as applied to $S$-sorted sets, the concept of support of an $S$-sorted set, and its properties, the notion of $S$-sorted equivalence on an $S$-sorted set, the quotient $S$-sorted set of an $S$-sorted set by an $S$-sorted equivalence on it, the usual set-theoretic operations on the $S$-sorted sets, and the notion of family of $S$-sorted sets with constant support.

Afterwards, for a set (of sorts) $S$ in $\ensuremath{\boldsymbol{\mathcal{U}}}$, we define the notion of $S$-sorted signature. Next, for an $S$-sorted signature $\Sigma$, we define the concepts of $\Sigma$-algebra, $\Sigma$-homomorphism (or, to abbreviate, homomorphism) from a $\Sigma$-algebra to another, and the corresponding category. Moreover, we define the notions of support of a $\Sigma$-algebra, of finite $\Sigma$-algebra, of family of $\Sigma$-algebras with constant support, and of subalgebra of a $\Sigma$-algebra, the construction of the product of a family of $\Sigma$-algebras, the concept of congruence on a $\Sigma$-algebra, and the construction of the quotient $\Sigma$-algebra of a $\Sigma$-algebra by a congruence on it.

From now on we make the following assumption: $S$ is a set of sorts in $\ensuremath{\boldsymbol{\mathcal{U}}}$, fixed once and for all.

\begin{definition}
The \emph{free monoid on} $S$, denoted by $\mathbf{S}^{\star}$, is $(S^{\star},\curlywedge,\lambda)$, where $S^{\star}$, the set of all \emph{words on} $S$, is $\bigcup_{n\in\mathbb{N}}\mathrm{Hom}(n,S)$, $\curlywedge$, the \emph{concatenation} of words on $S$, is the binary operation on $S^{\star}$ which sends a pair of words $(w,v)$ on $S$ to the mapping $w\curlywedge v$ from $\lvert w \rvert+\lvert v \rvert$ to $S$, where $\lvert w \rvert$ and $\lvert v \rvert$ are the lengths ($\equiv$ domains) of the mappings $w$ and $v$, respectively, defined as follows: $w\curlywedge v(i) = w_{i}$, if $0\leq i < \lvert w \rvert$; $w\curlywedge v(i) = v_{i-\lvert w \rvert}$, if $\lvert w \rvert\leq i < \lvert w \rvert+\lvert v \rvert$,
and $\lambda$, the \emph{empty word on} $S$, is the unique mapping from $0 = \varnothing$ to $S$.
\end{definition}

\begin{definition}
An $S$-\emph{sorted set} is a function $A = (A_{s})_{s\in S}$ from $S$ to $\ensuremath{\boldsymbol{\mathcal{U}}}$. If $A$ and $B$ are $S$-sorted sets, an $S$-\emph{sorted mapping from} $A$ \emph{to} $B$ is an $S$-indexed family $f = (f_{s})_{s\in S}$, where, for every $s$ in $S$, $f_{s}$ is a mapping from $A_{s}$ to  $B_{s}$. Thus, an $S$-sorted mapping from $A$ to $B$ is an element of $\prod_{s\in S}\mathrm{Hom}(A_{s}, B_{s})$. We denote by $\mathrm{Hom}(A,B)$ the set of all $S$-sorted mappings from $A$ to $B$. From now on, $\mathbf{Set}^{S}$ stands for the category of $S$-sorted sets and $S$-sorted mappings.
\end{definition}

\begin{definition}
Let $I$ be a set in $\ensuremath{\boldsymbol{\mathcal{U}}}$ and $(A^{i})_{i\in I}$ an $I$-indexed family of $S$-sorted sets. Then the \emph{product} of $(A^{i})_{i\in I}$, denoted by $\prod_{i\in I}A^{i}$, is the $S$-sorted set defined, for every $s\in S$, as $\left(\prod\nolimits_{i\in I}A^{i}\right)_{s} = \prod\nolimits_{i\in I}A^{i}_{s}$. Moreover, for every $i\in I$, the \emph{i-th canonical projection}, $\mathrm{pr}^{I,i} = (\mathrm{pr}^{I,i}_{s})_{s\in S}$, abbreviated to $\mathrm{pr}^{i} = (\mathrm{pr}^{i}_{s})_{s\in S}$ when this is unlikely to cause confusion, is the $S$-sorted mapping from  $\prod_{i\in I}A^{i}$ to $A^{i}$ which, for every $s\in S$, sends $(a_{i})_{i\in I}$ in $\prod_{i\in I}A^{i}_{s}$ to $a_{i}$ in $A^{i}_{s}$.
On the other hand, if $B$ is an $S$-sorted set and $(f^{i})_{i\in I}$ an $I$-indexed family of $S$-sorted mappings, where, for every $i\in I$, $f^{i}$ is an $S$-sorted mapping from $B$ to $A^{i}$, then we denote by $\left<f^{i}\right>_{i\in I}$ the unique $S$-sorted mapping $f$ from $B$ to $\prod_{i\in I}A^{i}$ such that, for every $i\in I$, $\mathrm{pr}^{i}\circ f = f^{i}$.

The remaining set-theoretic operations on $S$-sorted sets are defined in a similar way, i.e., componentwise.
\end{definition}

\begin{remark}
For a set $I$ in $\ensuremath{\boldsymbol{\mathcal{U}}}$ and an $I$-indexed family of $S$-sorted sets $(A^{i})_{i\in I}$, the ordered pair $(\prod_{i\in I}A^{i},(\mathrm{pr}^{i})_{i\in I})$ is a product of $(A^{i})_{i\in I}$ in $\mathbf{Set}^{S}$.
\end{remark}

\begin{definition}
We denote by $1^{S}$ or, to abbreviate, by $1$, the (standard) final $S$-sorted set of $\mathbf{Set}^{S}$, which is $1^{S} = (1)_{s\in S}$, and by $\varnothing^{S}$ the initial $S$-sorted set, which is $\varnothing^{S} = (\varnothing)_{s\in S}$.
\end{definition}

\begin{definition}
If $A$ and $B$ are $S$-sorted sets, then we will say that $A$ is a \emph{subset} of $B$, denoted by $A\subseteq B$, if, for every $s\in S$, $A_{s}\subseteq B_{s}$.
\end{definition}

\begin{definition}
Let $f,g\colon A\mor B$ be two $S$-sorted mappings. Then the \emph{equalizer} of $f$ and $g$, denoted by $\mathrm{Eq}(f,g)$, is the subset of $A$ defined, for every $s\in S$, as $\mathrm{Eq}(f,g)_{s} = \{a\in A_{s}\mid f_{s}(a)=g_{s}(a)\}$. Moreover, $\mathrm{eq}(f,g)$ is the canonical embedding of $\mathrm{Eq}(f,g)$ into $A$.
\end{definition}

\begin{remark}
For a parallel pair $f,g\colon A\mor B$ of $S$-sorted mappings, the ordered pair $(\mathrm{Eq}(f,g),\mathrm{eq}(f,g))$ is an equalizer of $f$ and $g$ in $\mathbf{Set}^{S}$.
\end{remark}

\begin{definition}
An $S$-sorted set $A$ is \emph{finite} if $\coprod A = \bigcup_{s\in S}(A_{s}\times \{s\})$ is finite. We say that $A$ is a \emph{finite} subset of $B$ if $A$ is finite and $A\subseteq B$.
\end{definition}

\begin{remark}
For an object $A$ of the topos $\mathbf{Set}^{S}$, are equivalent: (1) $A$ is finite, (2) $A$ is a finitary object of $\mathbf{Set}^{S}$, and (3) $A$ is a strongly finitary object of $\mathbf{Set}^{S}$.
In $\mathbf{Set}^{S}$ there is another notion of finiteness: An $S$-sorted set $A$ is $S$-finite if, and only if, for every $s\in S$, $A_{s}$ is finite. However, unless $S$ is finite, this notion of finiteness is not categorial.
\end{remark}

\begin{definition}
Let $A$ be an $S$-sorted set. Then the \emph{support of} $A$, denoted by $\mathrm{supp}_{S}(A)$, is the set $\{\,s\in S\mid A_{s}\neq \varnothing\,\}$.
\end{definition}

\begin{remark}
An $S$-sorted set $A$ is finite if, and only if, $\mathrm{supp}_{S}(A)$ is finite and, for every $s\in \mathrm{supp}_{S}(A)$, $A_{s}$ is finite.
\end{remark}


In the following proposition we gather together only those properties of the mapping   $\mathrm{supp}_{S}\colon\ensuremath{\boldsymbol{\mathcal{U}}}^{S}\mor \mathrm{Sub}(S)$, the support mapping for $S$, which sends an $S$-sorted set $A$ to $\mathrm{supp}_{S}(A)$, which will actually be used afterwards.

\begin{proposition}\label{propssupport}
Let $A$ and $B$ be two $S$-sorted sets, $I$ a set in $\ensuremath{\boldsymbol{\mathcal{U}}}$, and $(A^{i})_{i\in I}$ an $I$-indexed family of $S$-sorted sets. Then the following properties hold:
\begin{enumerate}
\item $\mathrm{Hom}(A,B)\neq \varnothing$ if, and only if,
      $\mathrm{supp}_{S}(A)\subseteq\mathrm{supp}_{S}(B)$. Therefore, if $A\subseteq B$, then $\mathrm{supp}_{S}(A)\subseteq\mathrm{supp}_{S}(B)$.

\item If from $A$ to $B$ there exists a surjective $S$-sorted mapping $f$, then we have that
      $\mathrm{supp}_{S}(A) = \mathrm{supp}_{S}(B)$.

\item $\mathrm{supp}_{S}(\prod_{i\in I}A^{i}) = \bigcap\nolimits_{i\in I}\mathrm{supp}_{S}(A^{i})$ (if $I = \varnothing$, we adopt the convention that $\bigcap\nolimits_{i\in I}\mathrm{supp}_{S}(A^{i}) = S$, since $\prod_{i\in \varnothing} A^{i}$ is $1 = (1)_{s\in S}$, the final object of $\mathbf{Set}^{S}$).
\end{enumerate}
\end{proposition}

\begin{remark}
The concept of support does not play any significant role in the case of the single-sorted algebras. Nevertheless, it (together with, among others, the notions of uniform algebraic closure operator on an $S$-sorted set, delta of Kronecker, subfinal $S$-sorted set, finite $S$-sorted set, and family of $S$-sorted sets with constant support) has turned to be essential to accomplish some investigations in the field of many-sorted algebras, e.g., those carried out in~\cite{cs04}--\cite{cs16}.
\end{remark}

In the following definition of the concept of family of $S$-sorted sets with constant support use will be made of the concept of support defined above.

\begin{definition}
Let $I$ be a set and $(A^{i})_{i\in I}$ an $I$-indexed family of $S$-sorted sets. We say that $(A^{i})_{i\in I}$ is a family of $S$-sorted sets with \emph{constant support} if, for every $i, j\in I$, $\mathrm{supp}_{S}(A^{i}) = \mathrm{supp}_{S}(A^{j})$.
\end{definition}

\begin{definition}
An $S$-\emph{sorted equivalence relation on} (or, to abbreviate, an $S$-\emph{sorted equivalence on}) an $S$-sorted set $A$ is an $S$-sorted relation $\Phi$ on $A$, i.e., a subset $\Phi = (\Phi_{s})_{s\in S}$ of the cartesian product $A\times A = (A_{s}\times A_{s})_{s\in S}$ such that, for every $s\in S$, $\Phi_{s}$ is an equivalence relation on $A_{s}$.

For an $S$-sorted equivalence relation $\Phi$ on $A$, $A/\Phi$, the $S$-\emph{sorted quotient set of} $A$ \emph{by} $\Phi$, is $(A_{s}/\Phi_{s})_{s\in S}$, and $\mathrm{pr}^{\Phi}\colon A\mor A/\Phi$, the \emph{canonical projection from} $A$ \emph{to} $A/\Phi$, is the $S$-sorted mapping $(\mathrm{pr}^{\Phi_{s}})_{s\in S}$, where, for every $s\in S$, $\mathrm{pr}^{\Phi_{s}}$ is the canonical projection from $A_{s}$ to $A_{s}/\Phi_{s}$ (which sends $x$ in $A_{s}$ to $\mathrm{pr}^{\Phi_{s}}(x) = [x]_{\Phi_{s}}$, the $\Phi_{s}$-equivalence class of $x$, in $A_{s}/\Phi_{s}$).


\end{definition}

\begin{remark}
Let $A$ be an $S$-sorted set and $\Phi\in\mathrm{Eqv}(A)$. Then, by Proposition~\ref{propssupport}, $\mathrm{supp}_{S}(A) = \mathrm{supp}_{S}(A/\Phi)$.
\end{remark}

We next recall the concept of kernel of an $S$-sorted mapping and the universal property of the $S$-sorted quotient set of an $S$-sorted set by an $S$-sorted equivalence on it

\begin{definition}
Let $f\colon A\mor B$ be an $S$-sorted mapping. Then the \emph{kernel} of $f$, denoted by  $\mathrm{Ker}(f)$, is the $S$-sorted relation defined, for every $s\in S$, as $\mathrm{Ker}(f)_{s} = \mathrm{Ker}(f_{s})$ (i.e., as the kernel pair of $f_{s}$).
\end{definition}

\begin{proposition}
If $f$ is an $S$-sorted mapping from $A$ to $B$, then we have that $\mathrm{Ker}(f)\in\mathrm{Eqv}(A)$. Moreover, given an $S$-sorted set $A$ and an $S$-sorted equivalence $\Phi$ on $A$, the pair $(\mathrm{pr}^{\Phi},A/\Phi)$ is such that (1) $\mathrm{Ker}(\mathrm{pr}^{\Phi}) = \Phi$, and (2) \emph{(universal property)} for every $S$-sorted mapping $f\colon A\mor B$, if $\Phi\subseteq \mathrm{Ker}(f)$, then there exists a unique $S$-sorted mapping $\mathrm{p}^{\Phi,\mathrm{Ker}(f)}$ from $A/\Phi$ to $B$ such that $f = \mathrm{p}^{\Phi,\mathrm{Ker}(f)}\circ \mathrm{pr}^{\Phi}$.
\end{proposition}



Following this we define, for the set of sorts $S$, the category of $S$-sorted signatures.

\begin{definition}\label{$S$-sorted signature}
An $S$-\emph{sorted signature} is a function $\Sigma$ from $S^{\star}\times S$ to $\ensuremath{\boldsymbol{\mathcal{U}}}$ which sends a pair $(w,s)\in S^{\star}\times S$ to the set $\Sigma_{w,s}$ of the
\emph{formal operations} of \emph{arity} $w$, \emph{sort} (or \emph{coarity}) $s$, and \emph{rank} (or \emph{biarity}) $(w,s)$.  Sometimes we will write $\sigma\colon w\mor s$ to indicate that the formal operation $\sigma$ belongs to
$\Sigma_{w,s}$.
\end{definition}

From now on we make the following assumption: $\Sigma$ stands for an $S$-sorted signature, fixed once and for all.

We next define the category of $\Sigma$-algebras.

\begin{definition}
The $S^{\star}\times S$-sorted set of the \emph{finitary operations on} an $S$-sorted set $A$ is $(\mathrm{Hom}(A_{w},A_{s}))_{(w,s)\in S^{\star}\times S}$, where, for every $w\in S^{\star}$, $A_{w} = \prod_{i\in \lvert w\rvert}A_{w_{i}}$, with $\lvert w\rvert$ denoting the length of the word $w$.
A \emph{structure of} $\Sigma$-\emph{algebra on} an $S$-sorted set  $A$ is a family $(F_{w,s})_{(w,s)\in S^{\star}\times S}$, denoted by $F$, where, for $(w,s)\in S^{\star}\times S$, $F_{w,s}$ is a mapping from $\Sigma_{w,s}$ to $\mathrm{Hom}(A_{w},A_{s})$. For a pair $(w,s)\in S^{\star}\times S$ and a formal operation $\sigma\in \Sigma_{w,s}$, in order to simplify the notation, the operation from $A_{w}$ to $A_{s}$ corresponding to $\sigma$ under $F_{w,s}$ will be written as $F_{\sigma}$ instead of $F_{w,s}(\sigma)$. A $\Sigma$-\emph{algebra} is a pair $(A,F)$, abbreviated to $\mathbf{A}$, where $A$ is an $S$-sorted set and $F$ a structure of $\Sigma$-algebra on $A$. A $\Sigma$-\emph{homomorphism} from $\mathbf{A}$ to $\mathbf{B}$, where $\mathbf{B} = (B,G)$, is a triple $(\mathbf{A},f,\mathbf{B})$, abbreviated to $f\colon \mathbf{A}\mor \mathbf{B}$, where $f$ is an $S$-sorted mapping from $A$ to $B$ such that, for every $(w,s)\in S^{\star}\times S$, every  $\sigma\in \Sigma_{w,s}$, and every $(a_{i})_{i\in \lvert w\rvert}\in A_{w}$, we have that
$
f_{s}(F_{\sigma}((a_{i})_{i\in \lvert w\rvert})) = G_{\sigma}(f_{w}((a_{i})_{i\in \lvert w\rvert})),
$
where $f_{w}$ is the mapping $\prod_{i\in \lvert w\rvert}f_{w_{i}}$ from $A_{w}$ to $B_{w}$ which sends $(a_{i})_{i\in \lvert w\rvert}$ in $A_{w}$ to $(f_{w_{i}}(a_{i}))_{i\in \lvert w\rvert}$ in $B_{w}$. We denote by $\mathbf{Alg}(\Sigma)$ the category of $\Sigma$-algebras and $\Sigma$-homomorphisms (or, to abbreviate, homomorphisms) and by $\mathrm{Alg}(\Sigma)$ the set of objects  of $\mathbf{Alg}(\Sigma)$.
\end{definition}


\begin{definition}
Let $\mathbf{A}$ be a $\Sigma$-algebra. Then the \emph{support of} $\mathbf{A}$, denoted by $\mathrm{supp}_{S}(\mathbf{A})$, is $\mathrm{supp}_{S}(A)$, i.e., the support of the underlying $S$-sorted set $A$ of $\mathbf{A}$.
\end{definition}

\begin{remark}
The set $\{\mathrm{supp}_{S}(\mathbf{A})\mid \mathbf{A}\in \mathrm{Alg}(\Sigma)\}$ is a closure system on $S$.
\end{remark}

\begin{definition}
Let $\mathbf{A}$ be a $\Sigma$-algebra. We say that $\mathbf{A}$ is \emph{finite} if $A$, the underlying $S$-sorted set of $\mathbf{A}$, is finite.
\end{definition}

We next define when a subset $X$ of the underlying $S$-sorted set $A$ of a $\Sigma$-algebra $\mathbf{A}$ is closed under an operation of $\mathbf{A}$, as well as when $X$ is a subalgebra of $\mathbf{A}$.

\begin{definition}\label{Subalg}
Let $\mathbf{A}$ be a $\Sigma$-algebra and $X\subseteq A$. Let $\sigma$ be such that $\sigma\colon w\mor s$, i.e., a formal operation in $\Sigma_{w,s}$. We say that $X$ is \emph{closed under the operation} $F_{\sigma}\colon A_{w}\mor A_{s}$ if, for every $a\in X_{w}$, $F_{\sigma}(a)\in X_{s}$. We say that $X$ is a \emph{subalgebra} of $\mathbf{A}$ if $X$ is closed under the operations of $\mathbf{A}$.
We also say, equivalently, that a $\Sigma$-algebra $\mathbf{X}$ is a \emph{subalgebra} of $\mathbf{A}$ if $X\subseteq A$ and the canonical embedding of $X$ into $A$ determines an embedding of $\mathbf{X}$ into $\mathbf{A}$.
\end{definition}

We now recall the concept of product of a family of $\Sigma$-algebras.

\begin{definition}
Let $I$ be a set in $\ensuremath{\boldsymbol{\mathcal{U}}}$ and $(\mathbf{A}^{i})_{i\in I}$ an $I$-indexed family of $\Sigma$-algebras, where, for every $i\in I$, $\mathbf{A}^{i} = (A^{i},F^{i})$. The \emph{product} of $(\mathbf{A}^{i})_{i\in I}$, denoted by $\prod_{i\in I}\mathbf{A}^{i}$, is the $\Sigma$-algebra $(\prod_{i\in I}A^{i},F)$ where, for every $\sigma\colon w\mor s$ in $\Sigma$, $F_{\sigma}$ sends $(a_{\alpha})_{\alpha\in \lvert w\rvert}$ in $(\prod_{i\in I}A^{i})_{w}$ to $(F^{i}_{\sigma}((a_{\alpha}(i))_{\alpha\in \lvert w\rvert}))_{i\in I}$ in $\prod_{i\in I}A^{i}_{s}$.
For every $i\in I$, the \emph{$i$-th canonical projection}, $\mathrm{pr}^{i} = (\mathrm{pr}^{i}_{s})_{s\in S}$, is the homomorphism from $\prod_{i\in I}\mathbf{A}^{i}$ to $\mathbf{A}^{i}$ which, for every $s\in S$, sends $(a_{i})_{i\in I}$ in $\prod_{i\in I}A^{i}_{s}$ to $a_{i}$ in $A^{i}_{s}$.
On the other hand, if $\mathbf{B}$ is a $\Sigma$-algebra and $(f^{i})_{i\in I}$ an $I$-indexed family of homomorphisms, where, for every $i\in I$, $f^{i}$ is a homomorphism from $\mathbf{B}$ to $\mathbf{A}^{i}$, then we denote by $\left<f^{i}\right>_{i\in I}$ the unique homomorphism $f$ from $\mathbf{B}$ to $\prod_{i\in I}\mathbf{A}^{i}$ such that, for every $i\in I$, $\mathrm{pr}^{i}\circ f = f^{i}$.
\end{definition}

In the following definition of the concept of family of $\Sigma$-algebras with constant support use will be made of the concept of an $I$-indexed family of $S$-sorted sets with constant support.

\begin{definition}
Let $I$ be a set and $(\mathbf{A}^{i})_{i\in I}$ an $I$-indexed family of $\Sigma$-algebras. We say that $(\mathbf{A}^{i})_{i\in I}$ is a family of $\Sigma$-algebras with \emph{constant support} if $(A^{i})_{i\in I}$, the underlying family of $S$-sorted sets of $(\mathbf{A}^{i})_{i\in I}$, is a family of $S$-sorted sets with constant support.
\end{definition}

Our next goal is to define the concepts of congruence on a $\Sigma$-algebra and of quotient of a $\Sigma$-algebra by a congruence on it. Moreover, we recall the notion of kernel of a homomorphism between $\Sigma$-algebras and the universal property of the quotient of a $\Sigma$-algebra by a congruence on it.

\begin{definition}
Let $\mathbf{A}$ be a $\Sigma$-algebra and $\Phi$ an $S$-sorted equivalence on $A$. We say that $\Phi$ is an
$S$-\emph{sorted congruence on} (or, to abbreviate, a \emph{congruence on}) $\mathbf{A}$ if, for every $(w,s)\in (S^{\star}-\{\lambda\})\times S$, every $\sigma\colon w\mor s$,
and every $a,b\in A_{w}$, if, for every $i\in \lvert w\rvert$, $(a_{i}, b_{i})\in\Phi_{w_{i}}$, then $(F_{\sigma}(a), F_{\sigma}(b))\in \Phi_{s}$.
\end{definition}

\begin{definition}
Let $\mathbf{A}$ be a $\Sigma$-algebra and $\Phi\in\mathrm{Cgr}(\mathbf{A})$. Then $\mathbf{A}/\Phi$, the \emph{quotient $\Sigma$-algebra} of $\mathbf{A}$ \emph{by} $\Phi$, is the $\Sigma$-algebra $(A/\Phi,F^{\mathbf{A}/\Phi})$, where, for every $\sigma\colon w\mor s$, the operation $F_{\sigma}^{\mathbf{A}/\Phi}\colon (A/\Phi)_{w}\mor A_{s}/\Phi_{s}$, also denoted, to simplify, by $F_{\sigma}$, sends $([a_{i}]_{\Phi_{w_{i}}})_{i\in\lvert w\rvert}$ in $(A/\Phi)_{w}$ to $[F_{\sigma}((a_{i})_{i\in \lvert w\rvert})]_{\Phi_{s}}$ in $A_{s}/\Phi_{s}$.  %
And $\mathrm{pr}^{\Phi}\colon \mathbf{A}\mor \mathbf{A}/\Phi$, the \emph{canonical projection from} $\mathbf{A}$ \emph{to} $\mathbf{A}/\Phi$, is the homomorphism determined by the $S$-sorted mapping $\mathrm{pr}^{\Phi}$ from $A$ to $A/\Phi$.
\end{definition}

\begin{proposition}
If $f$ is a homomorphism from $\mathbf{A}$ to $\mathbf{B}$, then $\mathrm{Ker}(f)\in\mathrm{Cgr}(\mathbf{A})$. Moreover, given a $\Sigma$-algebra $\mathbf{A}$ and a congruence $\Phi$ on $\mathbf{A}$, the pair $(\mathrm{pr}^{\Phi},\mathbf{A}/\Phi)$ is such that (1) $\mathrm{Ker}(\mathrm{pr}^{\Phi}) = \Phi$, and (2) \emph{(universal property)} for every homomorphism $f\colon\mathbf{A}\mor \mathbf{B}$, if $\Phi\subseteq \mathrm{Ker}(f)$, then there exists a unique homomorphism $\mathrm{p}^{\Phi,\mathrm{Ker}(f)}$ from $\mathbf{A}/\Phi$ to $\mathbf{B}$ such that $f = \mathrm{p}^{\Phi,\mathrm{Ker}(f)}\circ \mathrm{pr}^{\Phi}$.
\end{proposition}



\begin{proposition}
Let $f,g\colon\mathbf{A}\mor \mathbf{B}$ be two homomorphisms of $\Sigma$-algebras. Then the pair $(\mathbf{Eq}(f,g),\mathrm{eq}(f,g))$, with $\mathbf{Eq}(f,g)$ the subalgebra of $\mathbf{A}$ determined by the $S$-sorted set
$\mathrm{Eq}(f,g)=(\{a\in A_{s}\mid f_{s}(a)=g_{s}(a)\})_{s\in S}$, and $\mathrm{eq}(f,g)$ the canonical embedding in $\mathbf{A}$, is an equalizer of $f$ and $g$ in $\mathbf{Alg}(\Sigma)$.
\end{proposition}

We next define the concept of projective system of $\Sigma$-algebras and state the existence of the projective limit of a projective system of $\Sigma$-algebras. But before we start doing all that we recall that every preordered set $\mathbf{I} = (I,\leq)$ has a canonically associated category, also denoted by $\mathbf{I}$, whose set of objects is $I$ and whose set of morphisms is $\leq$, thus, for every $i,j\in I$, $\mathrm{Hom}(i,j) = \{(i,j)\}$, if $(i,j)\in \leq$, and $\mathrm{Hom}(i,j) = \varnothing$, otherwise.

\begin{definition}
Let $\mathbf{I}$ be a preordered set. A \emph{projective system} of $\Sigma$-algebras relative to $\mathbf{I}$ is a contravariant functor from (the category canonically associated to) $\mathbf{I}$ to $\mathbf{Alg}(\Sigma)$, i.e.,
an ordered pair $\boldsymbol{\mathcal{A}} = ((\mathbf{A}^{i})_{i\in I},(f^{j,i})_{(i,j)\in \leq})$ such that:
\begin{enumerate}
\item For every $i\in I$, $\mathbf{A}^{i}$ is a $\Sigma$-algebra.
\item For every $(i,j)\in\leq$, $f^{j,i}\colon\mathbf{A}^{j}\mor \mathbf{A}^{i}$.
\item For every $i\in I$, $f^{i,i} = \mathrm{id}_{\mathbf{A}^{i}}$.
\item For every $i,j,k\in I$, if $(i,j)\in \leq$ and $(j,k)\in
\leq$, then the following diagram commutes
$$\xymatrix@C=40pt@R=40pt{
\mathbf{A}^{k} \ar[r]^{f^{k,j}}
\ar[dr]_{f^{k,i}}
& \mathbf{A}^{j} \ar[d]^{f^{j,i}} \\
& \mathbf{A}^{i}
}
$$
\end{enumerate}
The homomorphisms $f^{j,i}\colon A^{j}\mor A^{i}$ are called the \emph{transition homomorphisms} of the projective system of $\Sigma$-algebras $\boldsymbol{\mathcal{A}}$ relative to $\mathbf{I}$.

A \emph{projective cone to} $\boldsymbol{\mathcal{A}}$ is an ordered pair $(\mathbf{L},(f^{i})_{i\in I})$ where $\mathbf{L}$ is a $\Sigma$-algebra and, for every $i\in I$, $f^{i}\colon \mathbf{L}\mor \mathbf{A}^{i}$, such that, for every $(i,j)\in \leq$, $f^{i} = f^{j,i}\circ f^{j}$. On the other hand, if $(\mathbf{L},(f^{i})_{i\in I})$ and $(\mathbf{M},(g^{i})_{i\in I})$ are two projective cones to $\boldsymbol{\mathcal{A}}$, then a \emph{morphism} from $(\mathbf{L},(f^{i})_{i\in I})$ to $(\mathbf{M},(g^{i})_{i\in I})$ is a homomorphism $h$ from $\mathbf{L}$ to $\mathbf{M}$ such that, for every $i\in I$, $f^{i} = g^{i}\circ h$.
A \emph{projective limit} of $\boldsymbol{\mathcal{A}}$ is a projective cone $(\mathbf{L},(f^{i})_{i\in I})$ to  $\boldsymbol{\mathcal{A}}$ such that, for every projective cone $(\mathbf{M},(g^{i})_{i\in I})$ to $\boldsymbol{\mathcal{A}}$, there exits a unique morphism from $(\mathbf{M},(g^{i})_{i\in I})$ to $(\mathbf{L},(f^{i})_{i\in I})$.
\end{definition}

%
%

\begin{proposition}
Let $\boldsymbol{\mathcal{A}}$ be a projective system of $\Sigma$-algebras relative to $\mathbf{I}$. Then we denote by
$\varprojlim_{\mathbf{I}}\boldsymbol{\mathcal{A}}$, the $\Sigma$-algebra determined by the subalgebra $\varprojlim_{\mathbf{I}}\mathcal{A}$ of $\prod_{i\in I}\mathbf{A}^{i}$, where  $\varprojlim_{\mathbf{I}}\mathcal{A}$ is defined as:
$$
  (\{x\in \textstyle\prod_{i\in I}A^{i}_{s}\mid \forall\, (i,j)\in \leq\,
  (f^{j,i}(\mathrm{pr}^{j}_{s}(x)) = \mathrm{pr}^{i}_{s}(x))\})_{s\in S}.
$$
On the other hand, for every $i\in I$, let $f^{i}$ be the composition
$\mathrm{pr}^{i}\circ\mathrm{inc}^{\varprojlim_{\mathbf{I}}\mathcal{A}}$, of the canonical embedding $\mathrm{inc}^{\varprojlim_{\mathbf{I}}\mathcal{A}}$ of $\varprojlim_{\mathbf{I}}\mathcal{A}$ into $\prod_{i\in I}A^{i}$ and the canonical projection $\mathrm{pr}^{i}$ from $\prod_{i\in I}A^{i}$ to $A^{i}$.
Then, for every $i\in I$, $f^{i}$ is a homomorphism from $\varprojlim_{\mathbf{I}}\boldsymbol{\mathcal{A}}$ to $\mathbf{A}^{i}$
and the pair $(\varprojlim_{\mathbf{I}}\boldsymbol{\mathcal{A}},(f^{i})_{i\in I})$ is a projective limit of $\boldsymbol{\mathcal{A}}$.
\end{proposition}


We next define the concept of inductive system of $\Sigma$-algebras and state the existence of the inductive limit of an inductive system of $\Sigma$-algebras.

\begin{definition}
Let $\mathbf{I}$ be an upward directed preordered set. An \emph{inductive system} of $\Sigma$-algebras relative to $\mathbf{I}$ is a covariant functor from (the category canonically associated to) $\mathbf{I}$ to $\mathbf{Alg}(\Sigma)$, i.e., an ordered pair $\boldsymbol{\mathcal{A}} = ((\mathbf{A}^{i})_{i\in I},(f^{i,j})_{(i,j)\in\leq})$ such that
\begin{enumerate}
\item For every $i\in I$, $\mathbf{A}^{i}$ is a $\Sigma$-algebra.

\item For every $(i,j)\in\leq$, $f^{i,j}\colon\mathbf{A}^{i}\mor \mathbf{A}^{j}$.

\item For every $i\in I$, $f^{i,i}=\mathrm{id}_{\mathbf{A}^{i}}$.

\item For every $i,j,k\in I$, if $i\leq j\leq k$, then the
following diagram commutes
$$\xymatrix@C=40pt@R=40pt{
\mathbf{A}^{i}
\ar[r]^{f^{i,j}}
\ar[rd]_{f^{i,k}} &
\mathbf{A}^{j}
\ar[d]^{f^{j,k}} \\
&
\mathbf{A}^{k}
}
$$
\end{enumerate}
The homomorphisms $f^{i,j}$ are called \emph{transition homomorphisms} of the inductive system of $\Sigma$-algebras $\boldsymbol{\mathcal{A}}$ relative to $\mathbf{I}$.

An \emph{inductive cone from} $\boldsymbol{\mathcal{A}}$ is an ordered pair $(\mathbf{L},(f^{i})_{i\in I})$ where $\mathbf{L}$ is a $\Sigma$-algebra and, for every $i\in I$, $f^{i}\colon \mathbf{A}^{i}\mor \mathbf{L}$, such that, for every $(i,j)\in \leq$, $f^{i} = f^{j}\circ f^{i,j}$. On the other hand, if $(\mathbf{L},(f^{i})_{i\in I})$ and $(\mathbf{M},(g_{i})_{i\in I})$ are two inductive cones from $\boldsymbol{\mathcal{A}}$, then a \emph{morphism} from $(\mathbf{L},(f^{i})_{i\in I})$ to  $(\mathbf{M},(g_{i})_{i\in I})$ is a homomorphism $h$ from $\mathbf{L}$ to $\mathbf{M}$ such that, for every $I\in I$, $g^{i} = h\circ f^{i}$.
An \emph{inductive limit} of $\boldsymbol{\mathcal{A}}$ is an inductive cone $(\mathbf{L},(f^{i})_{i\in I})$ from  $\boldsymbol{\mathcal{A}}$ such that, for every inductive cone $(\mathbf{M},(g^{i})_{i\in I})$ from $\boldsymbol{\mathcal{A}}$, there exits a unique morphism from $(\mathbf{L},(f^{i})_{i\in I})$ to $(\mathbf{M},(g^{i})_{i\in I})$.
\end{definition}

\begin{proposition}
Let $\boldsymbol{\mathcal{A}}$ be an inductive system of $\Sigma$-algebras relative to $\mathbf{I}$. Then we denote by
$\varinjlim_{\mathbf{I}}\boldsymbol{\mathcal{A}}$ the $\Sigma$-algebra which has as underlying $S$-sorted set
$\coprod_{i\in I}A^{i}/\Phi^{(\mathbf{I},\boldsymbol{\mathcal{A}})}$, where $\Phi^{(\mathbf{I},\boldsymbol{\mathcal{A}})}$ is the
$S$-equivalence on $\coprod_{i\in I}A^{i}$ defined as:
$$
\textstyle
\big(\{\,((a,i),(b,j))\in \big(\coprod_{i\in I}A^{i}_{s}\big)^{2}
\mid \exists k\in I ( k\geq i, j \And f^{i,k}_{s}(a) =
f^{j,k}_{s}(b)\,\}\big)_{s\in S},
$$
and, for every $(w,s)\in S^{\star}\times S$ and every $\sigma\in\Sigma_{w,s}$, as structural operation $F_{\sigma}$ from
$(\coprod_{i\in I}A^{i}/\Phi^{(\mathbf{I},\boldsymbol{\mathcal{A}})})_{w}$ to $\coprod_{i\in I}A^{i}_{s}/\Phi_{s}^{(\mathbf{I},\boldsymbol{\mathcal{A}})}$ corresponding to $\sigma$ that one defined by associating to an
$([(a_{\alpha},i_{\alpha})])_{ \alpha\in \lvert w \rvert}$ in $(\coprod_{i\in I}A^{i}/\Phi^{(\mathbf{I},\boldsymbol{\mathcal{A}})})_{w}$,
$[(F_{\sigma}^{k}(f^{i_{\alpha},k}(a_{\alpha})\mid {\alpha\in \lvert w \rvert}),k)]$ in $\coprod_{i\in
I}A^{i}_{s}/\Phi_{s}^{(\mathbf{I},\boldsymbol{\mathcal{A}})}$, where $k$ is an upper bound of $(i_{\alpha})_{ \alpha\in \lvert w  \rvert}$
in $\mathbf{I}$ and $F_{\sigma}^{k}$ the structural operation on $\mathbf{A}^{k}$ corresponding to $\sigma$.
On the other hand, for every $i\in I$, let $f^{i}$ be the composition
$\mathrm{pr}^{\Phi^{(\mathbf{I},\boldsymbol{\mathcal{A}})}}\circ\mathrm{inc}^{i}$, of the $S$-sorted mapping $\mathrm{inc}^{i}$ from $A^{i}$ to $\coprod_{i\in I}A^{i}_{s}$ and the $S$-sorted mapping $\mathrm{pr}^{\Phi^{(\mathbf{I},\boldsymbol{\mathcal{A}})}}$ from $\coprod_{i\in I}A^{i}_{s}$ to $\coprod_{i\in
I}A^{i}/\Phi^{(\mathbf{I},\boldsymbol{\mathcal{A}})}$.
Then, for every $i\in I$, $f^{i}$ is a homomorphism from $\mathbf{A}^{i}$ to $\varinjlim_{\mathbf{I}}\boldsymbol{\mathcal{A}}$
and the pair $(\varinjlim_{\mathbf{I}}\boldsymbol{\mathcal{A}},(f^{i})_{i\in I})$ is an inductive limit of $(\mathbf{I},\boldsymbol{\mathcal{A}})$.
\end{proposition}


In the single-sorted case, as in the many-sorted case, to calculate the inductive limit of an inductive system of
$\Sigma$-algebras, we can suppress from the inductive system those $\Sigma$-algebras which are initial, i.e., which have $\varnothing$ as underlying set.

\begin{remark}
Let $\boldsymbol{\mathcal{A}}$ be an inductive system of $\Sigma$-algebras relative to $\mathbf{I}$ and let $J$ be the subset of $I$ defined as
$$
 J = \{\,i\in I\mid A^{i}\neq (\varnothing)_{s\in S} \,\}.
$$
Then $\mathbf{J} = (J,\leq)$ is a directed preordered set (if $i,j\in J$, then $A^{i}\neq (\varnothing)_{s\in S}$, $A^{j}\neq
(\varnothing)_{s\in S}$, and there exists a $k\in I$ such that $k\geq i, j$, hence we have the homomorphisms $f^{i,k}$ from
$\mathbf{A}^{i}$ to  $\mathbf{A}^{k}$ and $f^{j,k}$ from $\mathbf{A}^{j}$ to  $\mathbf{A}^{k}$, therefore $A^{k}\neq
(\varnothing)_{s\in S}$, and, consequently, $k\in J$). Moreover, by definition, it is easy to see that $\varinjlim_{\mathbf{I}}\boldsymbol{\mathcal{A}}$ is the same as  $\varinjlim_{\mathbf{J}}\boldsymbol{\mathcal{A}}\!\!\upharpoonright\!\!J$. Therefore, to calculate the inductive limit of an inductive system of $\Sigma$-algebras, we can suppress from the inductive system those $\Sigma$-algebras which are initial, i.e., which have as underlying $S$-sorted set $(\varnothing)_{s\in S}$.
\end{remark}

Moreover, as it is well-known, for \emph{single-sorted} algebras, the inductive limit of an inductive system of \emph{nonempty}
$\Sigma$-algebras $\boldsymbol{\mathcal{A}}$ relative to $\mathbf{I}$ can be obtained, alternative, but equivalently, as a quotient algebra
$\mathbf{C}/{\equiv}$, where $\mathbf{C}$ is the subalgebra of $\prod_{i\in I}\mathbf{A}_{i}$ determined by the set $C$ of all those choice functions for $(A_{i})_{i\in I}$ which are \emph{eventually consistent}, i.e., by
$$
\textstyle
C = \{x\in\prod_{i\in I}A_{i}\mid\exists k\in I\,\forall j\geq i\geq k\, (f_{i,j}(x_{i}) = x_{j})\}
$$
and $\equiv$ the congruence on $\mathbf{C}$ defined as
$$
x\equiv y\text{ if and only if }\exists k\in I\,\forall i\geq k\, (x_{i} = y_{i}).
$$

However, for a set of sorts  $S$ such that $\mathrm{card}(S)\geq 2$, one can easily find $S$-sorted signatures $\Sigma$ and
$\Sigma$-algebras $\mathbf{A}$ such that
\begin{enumerate}
\item $\mathbf{A}$ is non-initial, i.e., such that the underlying $S$-sorted set is different from $(\varnothing)_{s\in S}$, but
\item $\mathbf{A}$ is \emph{globally empty}, i.e., such that there is not any homomorphism from $\mathbf{1}$, the final $\Sigma$-algebra, to $\mathbf{A}$.
\end{enumerate}
This fact has as a consequence that the above mentioned alternative construction of the inductive limit can not be applied
without qualification in the many-sorted case, because the suppression of every occurrence of the initial $\Sigma$-algebra in
a direct system does not have any effect on the elimination of those $\Sigma$-algebra which are non-initial but globally empty.

\begin{proposition}$[\,$\cite{cs16}, Prop.~2.5$\,]$\label{isomorfismoAlgebrasSoporteConstante}
Let $\boldsymbol{\mathcal{A}}$ be an inductive system of $\Sigma$-algebras relative to $\mathbf{I}$, $\mathbf{C}$ the subalgebra of $\prod_{i\in I}\mathbf{A}^{i}$ determined by the $S$-sorted set $C$ of $\prod_{i\in I}\mathbf{A}^{i}$ defined, for every $s\in S$, as follows
$$
\textstyle
C_{s}=\{x\in\prod_{i\in I}A^{i}_{s} \mid \exists\, k\in I,\;
\forall\, j\geq i\geq k,\; f^{i,j}_{s}(x_{i}) = x_{j}   \},
$$
and let $\equiv$ be the congruence on $\mathbf{C}$ defined, for every $s\in S$, as follows
$$
x\equiv_{s} y \text{ if and only if } \exists\, k\in I,\; \forall\, i\geq k,\; x_{i}=y_{i}.
$$
Then $(\mathbf{A}^{i})_{i\in I}$ is a family of $\Sigma$-algebras with constant support if and only if $\mathbf{C}/{\equiv}$ is isomorphic to $\varinjlim_{\mathbf{I}}\boldsymbol{\mathcal{A}}$.
\end{proposition}

The usual definitions of reduced products and ultraproducts for single-sorted algebras have an immediate translation for many-sorted algebras. However, some characterizations of such constructions are not valid for arbitrary families of many-sorted algebras, although they are valid for those families who have the additional property of having constant support.

\begin{definition}
Let $I$ be a nonempty set, $\mathcal{F}$ a filter on $I$, and  $(\mathbf{A}^{i})_{i\in I}$ a family of $\Sigma$-algebras. Then  $\boldsymbol{\mathcal{F}} = (\mathcal{F},\leq) = (\mathcal{F},\supseteq)$ is a nonempty upward directed preordered set and $\boldsymbol{\mathcal{A}}(\mathcal{F}) = ((\mathbf{A}(J))_{J\in \mathcal{F}},(\mathrm{p}^{K,J})_{K\leq J})$, where, for every $J\in \mathcal{F}$, $\mathbf{A}(J) = \prod_{j\in J}\mathbf{A}^{j}$ and, for every $J,K\in \mathcal{F}$ such that $K\supseteq J$, $\mathrm{p}^{K,J}$ denotes the unique $\Sigma$-homomorphism $\langle\mathrm{pr}^{K,j}\rangle_{j\in J}\colon\prod_{k\in K}\mathbf{A}^{k}\mor \prod_{j\in J}\mathbf{A}^{j}$ such that, for every $j\in J$, $\mathrm{pr}^{J,j}\circ \langle\mathrm{pr}^{K,j}\rangle_{j\in J} = \mathrm{pr}^{K,j}$, is an inductive system of $\Sigma$-algebras relative to $\boldsymbol{\mathcal{F}}$. The underlying $\Sigma$-algebra of the inductive limit $(\varinjlim_{\boldsymbol{\mathcal{F}}}\boldsymbol{\mathcal{A}}(\mathcal{F}),(\mathrm{p}^{J})_{J\in \mathcal{F}})$ of $\boldsymbol{\mathcal{A}}(\mathcal{F})$, also denoted by $\prod^{\mathcal{F}}_{i\in I}\mathbf{A}^{i}$, is called the \emph{reduced product} of $(\mathbf{A}^{i})_{i\in I}$ \emph{relative to} $\mathcal{F}$. If $\mathcal{F}$ is an ultrafilter on $I$, then the  underlying $\Sigma$-algebra of the inductive limit of the corresponding inductive system $\boldsymbol{\mathcal{A}}(\mathcal{F})$ is called the \emph{ultraproduct} of $(\mathbf{A}^{i})_{i\in I}$ \emph{relative to} $\mathcal{F}$.
\end{definition}

\begin{proposition}$[\,$\cite{cs16}, Prop.~2.7$\,]$
Let $I$ be a nonempty set, $\mathcal{F}$ a filter on $I$, and $(\mathbf{A}^{i})_{i\in I}$ a family of $\Sigma$-algebras. Then the  $S$-sorted relation $\equiv^{\mathcal{F}}$ in $\prod_{i\in I}A^{i}$, defined, for every $s\in S$, as follows %
$$
a\equiv^{\mathcal{F}}_{s}b\text{ if, and only if, } \mathrm{Eq}(a,b)\in \mathcal{F},
$$
where $\mathrm{Eq}(a,b)=\{i\in I\mid a_{i}=b_{i}\}$ is the equalizer of $a$ and $b$, is a congruence on $\prod_{i\in I}\mathbf{A}^{i}$.
\end{proposition}

\begin{proposition}$[\,$\cite{cs16}, Prop.~2.8$\,]$
Let $I$ be a nonempty set, $J$ a nonempty subset of $I$, $\mathcal{F}$ the principal filter on $I$ generated by $J$, and $(\mathbf{A}^{i})_{i\in I}$ a family of $\Sigma$-algebras. If $(\mathbf{A}^{i})_{i\in I}$ is a family with constant support, then $\prod_{i\in I}\mathbf{A}^{i}/{\equiv}^{\mathcal{F}}\cong \prod_{j\in J}\mathbf{A}^{j}$.
\end{proposition}

As it is well known, the reduced product of a family of single-sorted algebras is isomorphic to a quotient of the product of the family. However, when considering systems of many-sorted algebras, this representation is valid only for systems of many-sorted algebras with constant support.

\begin{lemma}\label{ConstSuppDerivedFamilyMSSet}
Let $I$ be a nonempty set and $\mathcal{F}$ a filter on $I$. If $(A^{i})_{i\in I}$ is an $I$-indexed family of $S$-sorted sets with constant support, then, for every $i\in I$ and every $J\in \mathcal{F}$, $\mathrm{supp}_{S}(A^{i}) = \mathrm{supp}_{S}(A(J))$, where $A(J)$ is the underlying $S$-sorted of $\mathbf{A}(J)$. Therefore $(A(J))_{J\in \mathcal{F}}$ is an $\mathcal{F}$-indexed family of $S$-sorted sets with constant support, i.e., for every $J,K\in\mathcal{F}$, $\mathrm{supp}_{S}(A(J)) = \mathrm{supp}_{S}(A(K))$.
\end{lemma}

\begin{proof}
Let $i$ be an element of $I$ and $J\in \mathcal{F}$. Then, by definition of $A(J)$, by Proposition~\ref{propssupport}, and by hypothesis, we have that $\mathrm{supp}_{S}(A(J)) = \bigcap_{j\in J}\mathrm{supp}_{S}(A^{j}) = \mathrm{supp}_{S}(A^{j})$, for every $j\in J$. But, by hypothesis, $\mathrm{supp}_{S}(A^{i}) = \mathrm{supp}_{S}(A^{j})$. Hence $\mathrm{supp}_{S}(A^{i}) = \mathrm{supp}_{S}(A(J))$. From this it follows, immediately, that $(A(J))_{J\in \mathcal{F}}$ is an $\mathcal{F}$-indexed family of $S$-sorted sets with constant support.
\end{proof}

\begin{proposition}$[\,$\cite{cs16}, Prop.~2.9$\,]$\label{CaracProdRed}
Let $I$ be a nonempty set, $\mathcal{F}$ a filter on $I$, and $(\mathbf{A}^{i})_{i\in I}$ a family of $\Sigma$-algebras. If  $(\mathbf{A}^{i})_{i\in I}$ is a family with constant support, then $\prod^{\mathcal{F}}_{i\in I}\mathbf{A}^{i}$ is isomorphic to $\prod_{i\in I}\mathbf{A}^{i}/{\equiv}^{\mathcal{F}}$.
\end{proposition}

\begin{remark}
Let $I$ be a nonempty set, $\mathcal{F}$ a filter on $I$, and $(\mathbf{A}^{i})_{i\in I}$ a family of $\Sigma$-algebras. If $\prod^{\mathcal{F}}_{i\in I}\mathbf{A}^{i} = \varinjlim_{\boldsymbol{\mathcal{F}}}\boldsymbol{\mathcal{A}}$ is isomorphic to $\prod_{i\in I}\mathbf{A}^{i}/{\equiv}^{\mathcal{F}}$ and $\mathcal{F}$ is such that, for every $s\in S$, $\{i\in I\mid s\in\mathrm{supp}_{S}(A^{i})\}\in \mathcal{F}$, then $(\mathbf{A}^{i})_{i\in I}$ is a family with constant support.
\end{remark}

\begin{corollary}\label{CaracUltraProd}
Let $I$ be a nonempty set, $\mathcal{F}$ an ultrafilter on $I$, and $(\mathbf{A}^{i})_{i\in I}$ a family of $\Sigma$-algebras. If $(\mathbf{A}^{i})_{i\in I}$ is a family with constant support, then $\prod^{\mathcal{F}}_{i\in I}\mathbf{A}^{i}$ is isomorphic to $\prod_{i\in I}\mathbf{A}^{i}/{\equiv}^{\mathcal{F}}$.
\end{corollary}

\section{A sufficient condition for a profinite $\Sigma$-algebra to be a retract of an ultraproduct of finite $\Sigma$-algebras.}

In this section, after recalling that for a nonempty upward directed preordered set $\mathbf{I}$ the set of all final sections of $\mathbf{I}$ is included in an ultrafilter on $I$ and stating that for a projective system of $S$-sorted sets  $\mathcal{A} = ((A^{i})_{i\in I},(f^{j,i})_{(i,j)\in \leq})$ relative to $\mathbf{I}$ and a filter $\mathcal{F}$ on $I$ such that the filter of the final sections of $\mathbf{I}$ is contained in $\mathcal{F}$, if the $I$-indexed family of $S$-sorted sets $(A^{i})_{i\in I}$ is with constant support, then the derived family $(A(J))_{J\in \mathcal{F}}$ is an $\mathcal{F}$-indexed family of $S$-sorted sets with constant support, we prove that if $\mathbf{A} = \varprojlim_{\mathbf{I}}\boldsymbol{\mathcal{A}}$ is a profinite $\Sigma$-algebra, where $\boldsymbol{\mathcal{A}}$ is a projective system of finite $\Sigma$-algebras relative to $\mathbf{I}$ with $\boldsymbol{\mathcal{A}} = ((\mathbf{A}^{i})_{i\in I},(f^{j,i})_{(i,j)\in \leq})$, and the $I$-indexed family of $\Sigma$-algebras $(\mathbf{A}^{i})_{i\in I}$ is with constant support, then $\mathbf{A}$ is a retract of $\prod_{i\in I}\mathbf{A}^{i}/\equiv^{\mathcal{F}}$.
\begin{assumption}
From now on we assume all preordered sets to be nonempty and upward directed.
\end{assumption}

\begin{proposition}
Let $\mathbf{I}$ be a preordered set. Then the subset $\{\Uparrow\!i\mid i\in I\}$ of $\mathrm{Sub}(I)$, where, for every $i\in I$, $\Uparrow\!i = \{j\in I\mid i\leq j\}$, the final section at $i$ of $\mathbf{I}$, is a filter basis on $I$, i.e., $\{\Uparrow\!i\mid i\in I\}\neq \varnothing$, $\varnothing\not\in \{\Uparrow\!i\mid i\in I\}$, and, for every $i, j\in I$ there exists a $k\in I$ such that  $\Uparrow\!k\subseteq \Uparrow\!i\cap\Uparrow\!j$.
\end{proposition}

We recall that for a preordered set $\mathbf{I}$, and according to the standard definition, the filter on $I$ generated by the filter basis $\{\Uparrow\!i\mid i\in I\}$ on $I$, which is called the filter of the final sections of $\mathbf{I}$ or the Fr\'{e}chet filter of $\mathbf{I}$, is
$$
\textstyle
\{I\}\cup \{J\subseteq I\mid \exists\, n\in \mathbb{N}-1\,\exists\,(i_{\alpha})_{\alpha\in n}\in I^{n}\,(\bigcap_{\alpha\in n}\Uparrow\!i_{\alpha}\subseteq J)\},
$$
which, on the basis of the above assumption, is precisely $\{J\subseteq I\mid \exists\,i\in I\,(\Uparrow\!i\subseteq J)\}$.
Moreover, since every filter $\mathcal{F}$ on a nonempty set $I$ is contained in an ultrafilter on $I$, it follows that $\{\Uparrow\!i\mid i\in I\}$ is contained in an ultrafilter on $I$.


From Lemma~\ref{ConstSuppDerivedFamilyMSSet} we obtain the following proposition.

\begin{proposition}
Let $\mathbf{I}$ be a preordered set and $\mathcal{F}$ a filter on $I$ such that the filter of the final sections of $\mathbf{I}$ is contained in $\mathcal{F}$. If $(A^{i})_{i\in I}$ is an $I$-indexed family of $S$-sorted sets with constant support, then, for every $i\in I$ and every $J\in \mathcal{F}$, $\mathrm{supp}_{S}(A^{i}) = \mathrm{supp}_{S}(A(J))$. Therefore $(A(J))_{J\in \mathcal{F}}$ is an $\mathcal{F}$-indexed family of $S$-sorted sets with constant support.
\end{proposition}

\begin{remark}
It is not true, in general, that if there are $j,k\in I$ such that $\mathrm{supp}_{S}(A^{j})\neq \mathrm{supp}_{S}(A^{k})$, then there are $J,K\in \mathcal{F}$ such that $\mathrm{supp}_{S}(A(J)) \neq \mathrm{supp}_{S}(A(K))$ or, what is equivalent, that if $(A(J))_{J\in \mathcal{F}}$ is an $\mathcal{F}$-indexed family of $S$-sorted sets with constant support, then $(A^{i})_{i\in I}$ is an $I$-indexed family of $S$-sorted sets with constant support. (This would be, trivially, fulfilled, e.g., if $(A(J))_{J\in \mathcal{F}}$ were an $\mathcal{F}$-indexed family of $S$-sorted sets with constant support and, for every $i\in I$ and every $j\in \Uparrow\!i$, $\mathrm{supp}_{S}(A^{i})\subseteq \mathrm{supp}_{S}(A^{j})$.) As an example, consider $S = \mathbb{N}$, $I = \mathbb{N}$, $\mathcal{F}$ the Fr\'{e}chet filter on $\mathbb{N}$, and $(A^{n})_{n\in \mathbb{N}}$ the $\mathbb{N}$-indexed family of $\mathbb{N}$-sorted sets, where, for every $n\in \mathbb{N}$, the $\mathbb{N}$-sorted set $A^{n} = (A^{n}_{m})_{m\in \mathbb{N}}$ is such that, for every $m\in \mathbb{N}$, $A^{n}_{m} = \varnothing$, if $n\neq m$, and $A^{n}_{m} = 1 = \{0\}$, otherwise.
\end{remark}

\begin{proposition}
Let $\mathbf{I}$ be a preordered set, $\mathcal{F}$ a filter on $I$ such that the filter of the final  sections of $\mathbf{I}$ is contained in $\mathcal{F}$, and $(A^{i})_{i\in I}$ an $I$-indexed family of $S$-sorted sets. Then the following assertions are equivalent:
\begin{enumerate}
\item $(A^{i})_{i\in I}$ is an $I$-indexed family of $S$-sorted sets with constant support.
\item For every $i\in I$ and every $J\in \mathcal{F}$, $\mathrm{supp}_{S}(A^{i}) = \mathrm{supp}_{S}(A(J))$.
\end{enumerate}
\end{proposition}

\begin{proof}
Since it is easy to check that (1) entails (2), we restrict ourselves to show that (2) entails (1).
Let us suppose that, for every $i\in I$ and every $J\in \mathcal{F}$, $\mathrm{supp}_{S}(A^{i}) = \mathrm{supp}_{S}(A(J))$. To prove that $(A^{i})_{i\in I}$ is an $I$-indexed family of $S$-sorted sets with constant support, let $k$ and $\ell$ be elements of $I$. Then we have that $\mathrm{supp}_{S}(A(\Uparrow\!k)) = \mathrm{supp}_{S}(A^{\ell})$. Hence, by Proposition~\ref{propssupport}, $\mathrm{supp}_{S}(A^{\ell})\subseteq \mathrm{supp}_{S}(A^{k})$. By a similar argument, $\mathrm{supp}_{S}(A^{k}) \subseteq \mathrm{supp}_{S}(A^{\ell})$. Hence $\mathrm{supp}_{S}(A^{k}) = \mathrm{supp}_{S}(A^{\ell})$. Therefore $(A^{i})_{i\in I}$ is an $I$-indexed family of $S$-sorted sets with constant support.
\end{proof}

%
%

From Lemma~\ref{ConstSuppDerivedFamilyMSSet} we obtain the following proposition.

\begin{proposition}\label{ConstSuppDerivedFamilyAlg}
Let $\mathbf{I}$ be a preordered set, $\mathcal{A} = ((A^{i})_{i\in I},(f^{j,i})_{(i,j)\in \leq})$ a projective system of $S$-sorted sets relative to $\mathbf{I}$, and $\mathcal{F}$ a filter on $I$ such that the filter of the final sections of $\mathbf{I}$ is contained in $\mathcal{F}$. If the $I$-indexed family of $S$-sorted sets $(A^{i})_{i\in I}$ is with constant support, then $(A(J))_{J\in \mathcal{F}}$ is an $\mathcal{F}$-indexed  family of $S$-sorted sets with constant support.
\end{proposition}

\begin{remark}
Let $\mathbf{I}$ be a preordered set, $\mathcal{A} = ((A^{i})_{i\in I},(f^{j,i})_{(i,j)\in \leq})$ a projective system of $S$-sorted sets, and $\mathcal{F}$ a filter on $I$ such that the filter of the final sections of $\mathbf{I}$ is contained in $\mathcal{F}$. If, for every $(i,j)\in \leq$, $f^{j,i}$ is surjective, then, by Proposition~\ref{propssupport} and taking into account that $\mathbf{I}$ is upward directed, $(A^{i})_{i\in I}$ is an $I$-indexed family of $S$-sorted sets with constant support.
\end{remark}



\begin{definition}
Let $\mathbf{A}$ be a $\Sigma$-algebra. We call $\mathbf{A}$ a \emph{profinite} $\Sigma$-algebra if it is a projective limit of a projective system of finite $\Sigma$-algebras.
\end{definition}

\begin{proposition}\label{MS Mariano and Miraglia}
Let $\mathbf{I}$ be a preordered set and $\mathcal{F}$ an ultrafilter on $I$ such that the filter basis $\{\Uparrow\!i\mid i\in I\}$ on $I$ is contained in $\mathcal{F}$. If $\mathbf{A} = \varprojlim_{\mathbf{I}}\boldsymbol{\mathcal{A}}$ is a profinite $\Sigma$-algebra, where $\boldsymbol{\mathcal{A}}$ is a projective system of finite $\Sigma$-algebras relative to $\mathbf{I}$ with $\boldsymbol{\mathcal{A}} = ((\mathbf{A}^{i})_{i\in I},(f^{j,i})_{(i,j)\in \leq})$, and the $I$-indexed family of finite  $\Sigma$-algebras $(\mathbf{A}^{i})_{i\in I}$ is with constant support, then $\mathbf{A}$ is a retract of $\prod_{i\in I}\mathbf{A}^{i}/\equiv^{\mathcal{F}}$.
\end{proposition}

\begin{proof}
By hypothesis, $(\mathbf{A}^{i})_{i\in I}$ is an $I$-indexed family of $\Sigma$-algebras with constant support, hence, by Proposition~\ref{ConstSuppDerivedFamilyAlg}, $(A(J))_{J\in \mathcal{F}}$ is an $\mathcal{F}$-indexed  family of $S$-sorted sets with constant support. Thus, by Corollary~\ref{CaracUltraProd}, $\prod_{i\in I}\mathbf{A}^{i}/{\equiv}^{\mathcal{F}}$ is isomorphic to $\prod^{\mathcal{F}}_{i\in I}\mathbf{A}^{i}$ which, we recall, is $\varinjlim_{\boldsymbol{\mathcal{F}}}\boldsymbol{\mathcal{A}}(\mathcal{F})$, the underlying $\Sigma$-algebra of the inductive limit $(\varinjlim_{\boldsymbol{\mathcal{F}}}\boldsymbol{\mathcal{A}}(\mathcal{F}),(\mathrm{p}^{J})_{J\in \mathcal{F}})$ of the inductive system $\boldsymbol{\mathcal{A}}(\mathcal{F})$ relative to $\boldsymbol{\mathcal{F}}$, where $\boldsymbol{\mathcal{F}}$ is $(\mathcal{F},\leq) = (\mathcal{F},\supseteq)$ and $\boldsymbol{\mathcal{A}}(\mathcal{F})$ is the ordered pair $((\mathbf{A}(J))_{J\in \mathcal{F}},(\mathrm{p}^{J,K})_{J\leq K})$. Therefore, since there exists a canonical embedding $\mathrm{in}^{\varprojlim_{\mathbf{I}}\boldsymbol{\mathcal{A}}}$ of $\mathbf{A} = \varprojlim_{\mathbf{I}}\boldsymbol{\mathcal{A}}$ into $\prod_{i\in I}\mathbf{A}^{i}$ and a canonical projection $\mathrm{pr}^{\equiv^{\mathcal{F}}}$ from $\prod_{i\in I}\mathbf{A}^{i}$ to $\prod_{i\in I}\mathbf{A}^{i}/{\equiv}^{\mathcal{F}}$, the problem comes down to show that there exists a homomorphism $h^{(\mathbf{I},\mathcal{F}),\boldsymbol{\mathcal{A}}}$ from $\prod^{\mathcal{F}}_{i\in I}\mathbf{A}^{i} = \varinjlim_{\boldsymbol{\mathcal{F}}}\boldsymbol{\mathcal{A}}(\mathcal{F})$ to $\mathbf{A} = \varprojlim_{\mathbf{I}}\boldsymbol{\mathcal{A}}$ such that the following diagram commutes:
$$\xymatrix@C=50pt@R=50pt{
\textstyle
\varprojlim_{\mathbf{I}}\boldsymbol{\mathcal{A}}
\ar[r]^{\mathrm{in}^{\varprojlim_{\mathbf{I}}\boldsymbol{\mathcal{A}}}}
\ar[rrd]_{\mathrm{id}_{\varprojlim_{\mathbf{I}}\boldsymbol{\mathcal{A}}}} &
\prod_{i\in I}\mathbf{A}^{i}
\ar[r]^-{\mathrm{pr}^{\equiv^{\mathcal{F}}}} &
\prod_{i\in I}\mathbf{A}^{i}/{\equiv}^{\mathcal{F}} \cong
\varinjlim_{\boldsymbol{\mathcal{F}}}\boldsymbol{\mathcal{A}}(\mathcal{F}) \ar[d]^{h^{(\mathbf{I},\mathcal{F}),\boldsymbol{\mathcal{A}}}} \\
{} & {} & \varprojlim_{\mathbf{I}}\boldsymbol{\mathcal{A}}
}
$$

To define $h^{(\mathbf{I},\mathcal{F}),\boldsymbol{\mathcal{A}}}$ (subject to satisfying the requirement just set out), we have to start by defining, for every $J\in \mathcal{F}$ and every $i\in I$, a homomorphism $h^{J,i}$ from $\mathbf{A}(J) = \prod_{j\in J}\mathbf{A}^{j}$ to $\mathbf{A}^{i}$ in such a way that, for every $J,K\in \mathcal{F}$ such that $K\supseteq J$, the homomorphisms $h^{J,i}$ from $\mathbf{A}(J)$ to $\mathbf{A}^{i}$ and $h^{K,i}$ from $\mathbf{A}(K)$ to $\mathbf{A}^{i}$ are compatible with the transition homomorphism $\mathrm{p}^{K,J}$ from $\mathbf{A}(K)$ to $\mathbf{A}(J)$. Afterwards, using the universal property of $\varinjlim_{\boldsymbol{\mathcal{F}}}\boldsymbol{\mathcal{A}}(\mathcal{F})$, we define a homomorphism $h^{i}$ from such an inductive limit to $\mathbf{A}^{i}$, for every $i\in I$. Finally, using the universal property of $\mathbf{A} = \varprojlim_{\mathbf{I}}\boldsymbol{\mathcal{A}}$, we obtain the desired homomorphism $h^{(\mathbf{I},\mathcal{F}),\boldsymbol{\mathcal{A}}}$ from $\varinjlim_{\boldsymbol{\mathcal{F}}}\boldsymbol{\mathcal{A}}(\mathcal{F})$ to $\mathbf{A} = \varprojlim_{\mathbf{I}}\boldsymbol{\mathcal{A}}$.

Let $J$ be an element of $\mathcal{F}$ and $i\in I$. We now proceed to define the homomorphism $h^{J,i} = (h^{J,i}_{s})_{s\in S}$ from $\mathbf{A}(J) = \prod_{j\in J}\mathbf{A}^{j}$ to $\mathbf{A}^{i}$.

For $s\in \mathrm{supp}_{S}(A^{i})$, $x\in A(J)_{s} = \prod_{j\in J}A^{j}_{s}$, and $y\in A^{i}_{s}$, let $V^{J,i,s}(x,y)$ be the subset of $J\cap\Uparrow\!i$ defined as follows:
$$
V^{J,i,s}(x,y) = \{j\in J\cap\Uparrow\!i\mid f^{j,i}_{s}(x_{j}) = y\}.
$$
The just stated definition is sound. In fact, $J\cap\Uparrow\!i\in \mathcal{F}$ since $\mathcal{F}$ is an ultrafilter such that  $\{\Uparrow\!i\mid i\in I\}\subseteq \mathcal{F}$ and $J\in \mathcal{F}$. Moreover, since, by hypothesis, $(\mathbf{A}^{i})_{i\in I}$ is a family of $\Sigma$-algebras with constant support we have that, for every $J\in \mathcal{F}$ and every $i\in I$, $\mathrm{supp}_{S}(A(J)) = \mathrm{supp}_{S}(A^{i})$.

For $J\in \mathcal{F}$, $i\in I$, $s\in \mathrm{supp}_{S}(A^{i})$, $x\in A(J)_{s} = \prod_{j\in J}A^{j}_{s}$, and $y,z\in A^{i}_{s}$, if $y\neq z$, then $V^{J,i,s}(x,y)\cap V^{J,i,s}(x,z) = \varnothing$. This follows from the fact that $f^{j,i}_{s}$ is, in particular, an $S$-sorted mapping.

We next prove that $J\cap\Uparrow\!i = \bigcup_{y\in A^{i}_{s}}V^{J,i,s}(x,y)$. It is obvious that $J\cap\Uparrow\!i$ contains  $\bigcup_{y\in A^{i}_{s}}V^{J,i,s}(x,y)$. Reciprocally, let $j$ be an element of $J\cap\Uparrow\!i$, then $i\leq j$ and for $y = f^{j,i}_{s}(x_{j})\in A^{i}_{s}$ we have that $j\in V^{J,i,s}(x,f^{j,i}_{s}(x_{j}))\subseteq \bigcup_{y\in A^{i}_{s}}V^{J,i,s}(x,y)$.

In what follows it is most useful to use a certain characterization of the notion of ultrafilter on a set. Specifically, a filter $\mathcal{G}$ on a nonempty set $I$ is an ultrafilter, i.e., a maximal filter, if, and only if, for every  $J,K\subseteq I$, if $J\cup K\in \mathcal{G}$, then $J\in \mathcal{G}$ or $K\in \mathcal{G}$. This characterization extends, by induction, up to  nonempty finite families of subsets of $I$. Moreover, we recall that $\varnothing$ does not belong to any filter. 

Now, as we have, on the one hand, that $\mathcal{F}$ is an ultrafilter such that $J\cap\Uparrow\!i\in \mathcal{F}$ and, on the other hand, that $J\cap\Uparrow\!i = \bigcup_{y\in A^{i}_{s}}V^{J,i,s}(x,y)$, that $A^{i}_{s}$ is finite, and that if  $y,z\in A^{i}_{s}$ are such that $y\neq z$, then $V^{J,i,s}(x,y)\cap V^{J,i,s}(x,z) = \varnothing$, we infer that there exists a unique $y\in A^{i}_{s}$ such that $V^{J,i,s}(x,y)\in \mathcal{F}$. Therefore, we define the mapping $h^{J,i}_{s}$ from $A(J)_{s} = \prod_{j\in J}A^{j}_{s}$ to $A^{i}_{s}$ by assigning to $x\in A(J)_{s}$ the unique $y\in A^{i}_{s}$ such that $V^{J,i,s}(x,y)\in \mathcal{F}$. Thus, for $x\in A(J)_{s}$ and $y\in A^{i}_{s}$, $h^{J,i}_{s}(x) = y$ if, and only if, $V^{J,i,s}(x,y)\in \mathcal{F}$.

Our next goal is to show that, for every $i\in I$ and every $J,K\in \mathcal{F}$, if $K\supseteq J$, then the homomorphism $\mathrm{p}^{K,J}$ from $\mathbf{A}(K)$ to $\mathbf{A}(J)$ is such that $h^{J,i}\circ \mathrm{p}^{K,J} = h^{K,i}$ and that $h^{J,i} = (h^{J,i}_{s})_{s\in S}$ is a homomorphism from $\mathbf{A}(J) = \prod_{j\in J}\mathbf{A}^{j}$ to $\mathbf{A}^{i}$.

To verify that $h^{J,i}\circ \mathrm{p}^{K,J} = h^{K,i}$, i.e., that, for every $s\in S$, $h^{J,i}_{s}\circ \mathrm{p}^{K,J}_{s} = h^{K,i}_{s}$, we should check that, for every $a\in A(K)_{s}$, $h^{J,i}_{s}(\mathrm{p}^{K,J}_{s}(a)) = h^{K,i}_{s}(a)$. But, for every $s\in S$, if $a\in A(K)_{s}$, then, by definition,  $\mathrm{p}^{K,J}_{s}(a) = a\!\!\upharpoonright\!\! J$, where $a\!\!\upharpoonright\!\! J$ is the restriction of $a$ to $J$. Therefore we should check that $h^{J,i}_{s}(a\!\!\upharpoonright\!\! J) = h^{K,i}_{s}(a)$. Let $y$ be $h^{J,i}_{s}(a\!\!\upharpoonright\!\! J)$, i.e., $y$ is the unique element of $A^{i}_{s}$ such that $V^{J,i,s}(a\!\!\upharpoonright\!\! J,y)\in \mathcal{F}$. Then it happens that
$$
V^{J,i,s}(a\!\!\upharpoonright\!\! J,y) \subseteq  V^{K,i,s}(a,y).
$$
Let $j$ be an element of $V^{J,i,s}(a\!\!\upharpoonright\!\! J,y) (= V^{J,i,s}(a\!\!\upharpoonright\!\! J,h^{J,i}_{s}(a\!\!\upharpoonright\!\! J)))$. Then $j\in J\cap\Uparrow\!i$ and $f^{j,i}_{s}((a\!\!\upharpoonright\!\! J)_{j}) = f^{j,i}_{s}(a_{j}) = y$. But, since $J\subseteq K$, we have that $J\cap\Uparrow\!i\subseteq K\cap\Uparrow\!i$. Therefore $j\in K\cap\Uparrow\!i$ and $f^{j,i}_{s}(a_{j}) = y$, i.e., $j\in V^{K,i,s}(a,y)$. Moreover, because $V^{J,i,s}(a\!\!\upharpoonright\!\! J,y)\in \mathcal{F}$, $V^{J,i,s}(a\!\!\upharpoonright\!\! J,y) \subseteq  V^{K,i,s}(a,y)$, and $\mathcal{F}$ is a filter, $V^{K,i,s}(a,y)\in \mathcal{F}$. From this it follows that $h^{K,i}_{s}(a) = y$. Therefore $h^{J,i}_{s}(a\!\!\upharpoonright\!\! J) = h^{K,i}_{s}(a)$ and, consequently, $h^{J,i}\circ \mathrm{p}^{K,J} = h^{K,i}$.

To show that $h^{J,i} = (h^{J,i}_{s})_{s\in S}$ is a homomorphism from $\mathbf{A}(J) = \prod_{j\in J}\mathbf{A}^{j}$ to  $\mathbf{A}^{i}$ we have to check that, for every $(w,s)\in S^{\star}\times S$, every $\sigma\in \Sigma_{w,s}$, and every  $(a_{\alpha})_{\alpha\in\lvert w \rvert}\in A(J)_{w} = (\prod_{j\in J}A^{j})_{w} = (\prod_{j\in J}A^{j}_{w_{0}})\times\cdots\times (\prod_{j\in J}A^{j}_{w_{\lvert w \rvert-1}})$, it happens that
$$
h^{J,i}_{s}(F^{\mathbf{A}(J)}_{\sigma}((a_{\alpha})_{\alpha\in\lvert w \rvert})) = F^{\mathbf{A}^{i}}_{\sigma}(h^{J,i}_{w_{0}}(a_{0}),\ldots,h^{J,i}_{w_{\lvert w \rvert-1}}(a_{\lvert w \rvert-1})).
$$
Let us recall that the structural operation $F^{\mathbf{A}(J)}_{\sigma}$ of $\mathbf{A}(J)$ is defined, for every $(a_{\alpha})_{\alpha\in\lvert w \rvert}\in A(J)_{w}$, as:
$$
F^{\mathbf{A}(J)}_{\sigma}((a_{\alpha})_{\alpha\in\lvert w \rvert}) =
(F^{\mathbf{A}^{j}}_{\sigma}((a_{\alpha}(j))_{\alpha\in\lvert w \rvert}))_{j\in J}.
$$

Now, for every $\alpha\in\lvert w \rvert$, we have the subset 
$$
V^{J,i,w_{\alpha}}(a_{\alpha},h^{J,i}_{w_{\alpha}}(a_{\alpha})) = \{j\in J\cap\Uparrow\!i\mid f^{j,i}_{w_{\alpha}}(a_{\alpha}(j)) = h^{J,i}_{w_{\alpha}}(a_{\alpha})\}.
$$
of $I$. But, for every $\alpha\in\lvert w \rvert$, we have that $V^{J,i,w_{\alpha}}(a_{\alpha},h^{J,i}_{w_{\alpha}}(a_{\alpha}))\in \mathcal{F}$. Thus, because $\mathcal{F}$ is a filter, we have that
$\bigcap_{\alpha\in\lvert w \rvert}V^{J,i,w_{\alpha}}(a_{\alpha},h^{J,i}_{w_{\alpha}}(a_{\alpha}))\in \mathcal{F}$. Moreover, we have the subset $V^{J,i,s}(F^{\mathbf{A}(J)}_{\sigma}((a_{\alpha})_{\alpha\in\lvert w \rvert}),
F^{\mathbf{A}^{i}}_{\sigma}((h^{J,i}_{w_{\alpha}}(a_{\alpha}))_{\alpha\in\lvert w \rvert}))$ of $I$, which, we recall, is
$$
\{j\in J\cap\Uparrow\!i\mid f^{j,i}_{s}(F^{\mathbf{A}^{j}}_{\sigma}((a_{\alpha}(j))_{\alpha\in\lvert w \rvert})) =
F^{\mathbf{A}^{i}}_{\sigma}((h^{J,i}_{w_{\alpha}}(a_{\alpha}))_{\alpha\in\lvert w \rvert})\}.
$$
Then it happens that
$$
\textstyle
\bigcap_{\alpha\in\lvert w \rvert}V^{J,i,w_{\alpha}}(a_{\alpha},h^{J,i}_{w_{\alpha}}(a_{\alpha}))\subseteq V^{J,i,s}(F^{\mathbf{A}(J)}_{\sigma}((a_{\alpha})_{\alpha\in\lvert w \rvert}),
F^{\mathbf{A}^{i}}_{\sigma}((h^{J,i}_{w_{\alpha}}(a_{\alpha}))_{\alpha\in\lvert w \rvert})).
$$
Let $j$ be an element of $\bigcap_{\alpha\in\lvert w \rvert}V^{J,i,w_{\alpha}}(a_{\alpha},h^{J,i}_{w_{\alpha}}(a_{\alpha}))$. Then, by definition, $i\leq j$ and, for every $\alpha\in \lvert w \rvert$, we have that $f^{j,i}_{w_{\alpha}}(a_{\alpha}(j)) = h^{J,i}_{w_{\alpha}}(a_{\alpha})$. But, $f^{j,i}$ is a homomorphism from $\mathbf{A}^{j}$ to $\mathbf{A}^{i}$, thus
\begin{align*}
f^{j,i}_{s}(F^{\mathbf{A}^{j}}_{\sigma}((a_{\alpha}(j))_{\alpha\in\lvert w \rvert})) &=  F^{\mathbf{A}^{i}}_{\sigma}(f^{j,i}_{w_{0}}(a_{0}(j)),\ldots,f^{j,i}_{w_{\lvert w \rvert-1}}(a_{\lvert w \rvert-1}(j))) \\
           &=  F^{\mathbf{A}^{i}}_{\sigma}(h^{J,i}_{w_{0}}(a_{0}),\ldots,h^{J,i}_{w_{\lvert w \rvert-1}}(a_{\lvert w \rvert-1})).
\end{align*}
Moreover, we have that
\begin{align*}
f^{j,i}_{s}(F^{\mathbf{A}^{J}}_{\sigma}((a_{\alpha})_{\alpha\in\lvert w \rvert})(j)) &=
f^{j,i}_{s}(F^{\mathbf{A}^{j}}_{\sigma}((a_{\alpha}(j))_{\alpha\in\lvert w \rvert}))  \\
         &=  F^{\mathbf{A}^{i}}_{\sigma}(h^{J,i}_{w_{0}}(a_{0}),\ldots,h^{J,i}_{w_{\lvert w \rvert-1}}(a_{\lvert w \rvert-1})).
\end{align*}
Therefore $j\in V^{J,i,s}(F^{\mathbf{A}(J)}_{\sigma}((a_{\alpha})_{\alpha\in\lvert w \rvert}),
F^{\mathbf{A}^{i}}_{\sigma}((h^{J,i}_{w_{\alpha}}(a_{\alpha}))_{\alpha\in\lvert w \rvert}))$. Hence, since $\mathcal{F}$ is a filter, we have that $V^{J,i,s}(F^{\mathbf{A}(J)}_{\sigma}((a_{\alpha})_{\alpha\in\lvert w \rvert}),
F^{\mathbf{A}^{i}}_{\sigma}((h^{J,i}_{w_{\alpha}}(a_{\alpha}))_{\alpha\in\lvert w \rvert}))\in \mathcal{F}$. So $h^{J,i} = (h^{J,i}_{s})_{s\in S}$ is a homomorphism from $\mathbf{A}(J) = \prod_{j\in J}\mathbf{A}^{j}$ to $\mathbf{A}^{i}$.

After having proved that, for every $i\in I$ and every $J,K\in \mathcal{F}$, if $K\supseteq J$, the homomorphism $\mathrm{p}^{K,J}$ from $\mathbf{A}(K)$ to $\mathbf{A}(J)$ is such that $h^{J,i}\circ \mathrm{p}^{K,J} = h^{K,i}$ and that $h^{J,i} = (h^{J,i}_{s})_{s\in S}$ is a homomorphism from $\mathbf{A}(J)$ to $\mathbf{A}^{i}$, we can assert, by the universal property of the inductive limit, that, for every $i\in I$, there exist a unique homomorphism $h^{i}$ from $\varinjlim_{\boldsymbol{\mathcal{F}}}\boldsymbol{\mathcal{A}}(\mathcal{F})$ to $\mathbf{A}^{i}$ such that, for every $J\in \mathcal{F}$, $h^{J,i} = h^{i}\circ \mathrm{p}^{J}$, where $\mathrm{p}^{J}$ is the canonical homomorphism from $\mathbf{A}(J)$ to  $\varinjlim_{\boldsymbol{\mathcal{F}}}\boldsymbol{\mathcal{A}}(\mathcal{F})$.

Our next goal is to show that, for every $i,k\in I$, if  $i\leq k$, then the homomorphisms $f^{k,i}\circ h^{k}$ and $h^{i}$ from $\varinjlim_{\boldsymbol{\mathcal{F}}}\boldsymbol{\mathcal{A}}(\mathcal{F})$ to $\mathbf{A}^{i}$ are equal.

To do this, we begin by showing that, for every $J\in \mathcal{F}$ and every $i,k\in I$, if  $i\leq k$, then $h^{J,i} = f^{k,i}\circ h^{J,k}$. Let us recall that, for every $s\in S$, the mapping $h^{J,i}_{s}$ from $A(J)_{s}$ to $A^{i}_{s}$ is defined by assigning to $x\in A(J)_{s}$ the unique $y\in A^{i}_{s}$ such that $V^{J,i,s}(x,y)\in \mathcal{F}$, where
$$
V^{J,i,s}(x,y) = \{j\in J\cap\Uparrow\!i\mid f^{j,i}_{s}(x_{j}) = y\}.
$$
It happens that $V^{J,k,s}(x,h^{J,k}_{s}(x))\subseteq V^{J,i,s}(x,f^{k,i}_{s}(h^{J,k}_{s}(x)))$. In fact, let $j$ be an element of $J\cap\Uparrow\!k$ such that $f^{j,k}_{s}(x_{j}) = h^{J,k}_{s}(x)$. Then, since $i\leq k$, we have that $j\in J\cap\Uparrow\!i$. It only remains to verify that $f^{j,i}_{s}(x_{j}) = f^{k,i}_{s}(h^{J,k}_{s}(x))$. But this follows from $f^{j,i} = f^{k,i}\circ f^{j,k}$ and $f^{j,k}_{s}(x_{j}) = h^{J,k}_{s}(x)$. However, since $V^{J,k,s}(x,h^{J,k}_{s}(x))\in \mathcal{F}$, we have that $V^{J,i,s}(x,f^{k,i}_{s}(h^{J,k}_{s}(x)))\in \mathcal{F}$. Thus, for every $s\in S$ and every $x\in A(J)_{s}$, $h^{J,i}_{s}(x) = f^{k,i}_{s}(h^{J,k}_{s}(x))$. Therefore $h^{J,i} = f^{k,i}\circ h^{J,k}$.

We are now in a position to show that, for $i\leq j$, $f^{k,i}\circ h^{k} = h^{i}$. In fact, we know that given $i,k\in I$ such that $i\leq k$, for every $J\in \mathcal{F}$, $h^{J,i} = f^{k,i}\circ h^{J,k}$, $h^{J,i} = h^{i}\circ \mathrm{p}^{J}$, and $h^{J,k} = h^{k}\circ \mathrm{p}^{J}$ or, what is equivalent, that the outer, the left and the right triangles of the following diagram commute:
$$\xymatrix@C=40pt@R=40pt{
 {} &  \mathbf{A}(J)\ar[d]_{\mathrm{p}^{J}}\ar@/_1pc/[ldd]_{h^{J,k}}\ar@/^1pc/[rdd]^{h^{J,i}} & {}   \\
 {} & \varinjlim_{\boldsymbol{\mathcal{F}}}\boldsymbol{\mathcal{A}}(\mathcal{F}) \ar[ld]_{h^{k}}
\ar[rd]^{h^{i}} & {}  \\
   \mathbf{A}_{k} \ar[rr]_{f^{k,i}} & {} & \mathbf{A}_{i}
    }
$$
Therefore $(f^{k,i}\circ h^{k})\circ \mathrm{p}^{J} = h^{i}\circ \mathrm{p}^{J}$. But any inductive limit is an (extremal epi)-sink, thus $f^{k,i}\circ h^{k} = h^{i}$.

After having proved that, for every $i,k\in I$, if $i\leq j$, then $f^{k,i}\circ h^{k} = h^{i}$, we can assert, by the universal property of the projective limit, that there exist a unique homomorphism $h^{(\mathbf{I},\mathcal{F}),\boldsymbol{\mathcal{A}}}$ from $\varinjlim_{\boldsymbol{\mathcal{F}}}\boldsymbol{\mathcal{A}}(\mathcal{F})$ to $\mathbf{A} = \varprojlim_{\mathbf{I}}\boldsymbol{\mathcal{A}}$ such that, for every $i\in I$, $f^{i}\circ h^{(\mathbf{I},\mathcal{F}),\boldsymbol{\mathcal{A}}} = h^{i}$, where $f^{i}$ is the canonical homomorphism from $\varprojlim_{\mathbf{I}}\boldsymbol{\mathcal{A}}$ to $\mathbf{A}^{i}$.

Finally, we proceed to show that $h^{(\mathbf{I},\mathcal{F}),\boldsymbol{\mathcal{A}}}\circ \mathrm{pr}^{\equiv^{\mathcal{F}}}\circ \mathrm{in}^{\varprojlim_{\mathbf{I}}\boldsymbol{\mathcal{A}}} = \mathrm{id}_{\varprojlim_{\mathbf{I}}\boldsymbol{\mathcal{A}}}$, where, we recall, $\mathrm{in}^{\varprojlim_{\mathbf{I}}\boldsymbol{\mathcal{A}}}$ is the canonical embedding of $\mathbf{A} = \varprojlim_{\mathbf{I}}\boldsymbol{\mathcal{A}}$ into $\prod_{i\in I}\mathbf{A}^{i}$ and $\mathrm{pr}^{\equiv^{\mathcal{F}}}$ the canonical projection from $\prod_{i\in I}\mathbf{A}^{i}$ to $\prod_{i\in I}\mathbf{A}^{i}/{\equiv}^{\mathcal{F}}$ which, we remark, coincides with $\mathrm{p}^{I}$, the canonical homomorphism from $\mathbf{A}(I) = \prod_{i\in I}\mathbf{A}^{i}$ to $\varinjlim_{\boldsymbol{\mathcal{F}}}\boldsymbol{\mathcal{A}}(\mathcal{F})$. But $\varprojlim_{\mathbf{I}}\boldsymbol{\mathcal{A}}$ is a projective limit and any projective limit is an (extremal mono)-source. Thus, to prove the above equality it suffices to prove that, for every $i\in I$, we have that
$$
f^{i}\circ (h^{(\mathbf{I},\mathcal{F}),\boldsymbol{\mathcal{A}}}\circ \mathrm{pr}^{\equiv^{\mathcal{F}}}\circ \mathrm{in}^{\varprojlim_{\mathbf{I}}\boldsymbol{\mathcal{A}}}) = f^{i}\circ \mathrm{id}_{\varprojlim_{\mathbf{I}}\boldsymbol{\mathcal{A}}} = f^{i}.
$$
We draw the following picture to provide a visual description of the current situtation.
$$\xymatrix@C=50pt@R=50pt{
\textstyle
\varprojlim_{\mathbf{I}}\boldsymbol{\mathcal{A}}
\ar[r]^{\mathrm{in}^{\varprojlim_{\mathbf{I}}\boldsymbol{\mathcal{A}}}}
\ar[rd]_{f^{i}} &
\mathbf{A}(I) = \prod_{i\in I}\mathbf{A}^{i}
\ar[r]^-{\mathrm{pr}^{\equiv^{\mathcal{F}}} = \mathrm{p}^{I}}
\ar @<1ex>[d]^-{h^{I,i}}\ar @<-1ex>[d]_-{\mathrm{pr}^{I,i}}&
\prod_{i\in I}\mathbf{A}^{i}/{\equiv}^{\mathcal{F}} \cong
\varinjlim_{\boldsymbol{\mathcal{F}}}\boldsymbol{\mathcal{A}}(\mathcal{F}) \ar[d]^{h^{(\mathbf{I},\mathcal{F}),\boldsymbol{\mathcal{A}}}}\ar[dl]_{h^{i}} \\
{} & \mathbf{A}^{i} & \varprojlim_{\mathbf{I}}\boldsymbol{\mathcal{A}}\ar[l]^{f^{i}}
}
$$
Let $i$ be an element of $I$. Then, as we have shown before, $f^{i}\circ h^{(\mathbf{I},\mathcal{F}),\boldsymbol{\mathcal{A}}} = h^{i}$ and $h^{i}\circ \mathrm{pr}^{\equiv^{\mathcal{F}}} = h^{i}\circ \mathrm{p}^{I} = h^{I,i}$. And, by definition of the canonical homomorphism $f^{i}$ of the projective limit $\varprojlim_{\mathbf{I}}\boldsymbol{\mathcal{A}}$, we have that $\mathrm{pr}^{I,i}\circ \mathrm{in}^{\varprojlim_{\mathbf{I}}\boldsymbol{\mathcal{A}}} = f^{i}$. Thus it only remains to prove that $h^{I,i}\circ \mathrm{in}^{\varprojlim_{\mathbf{I}}\boldsymbol{\mathcal{A}}} = \mathrm{pr}^{I,i}\circ \mathrm{in}^{\varprojlim_{\mathbf{I}}\boldsymbol{\mathcal{A}}}$. Let $s$ be an element of $S$ and $x$ an element of the $s$-th component of the underlying $S$-sorted set of $\varprojlim_{\mathbf{I}}\boldsymbol{\mathcal{A}}$. Then, taking into account that $\mathrm{in}^{\varprojlim_{\mathbf{I}}\boldsymbol{\mathcal{A}}}_{s}(x) = x$ and $\mathrm{pr}^{I,i}_{s}(\mathrm{in}^{\varprojlim_{\mathbf{I}}\boldsymbol{\mathcal{A}}}_{s}(x)) = x_{i}$, the sets
$$
V^{I,i,s}(\mathrm{in}^{\varprojlim_{\mathbf{I}}\boldsymbol{\mathcal{A}}}_{s}(x),
\mathrm{pr}^{I,i}_{s}(\mathrm{in}^{\varprojlim_{\mathbf{I}}\boldsymbol{\mathcal{A}}}_{s}(x))) = \{j\in I\cap\Uparrow\!i\mid f^{j,i}_{s}(x_{j}) = x_{i}\}
$$
and $\Uparrow\!i$ are, obviously, equal. But $\Uparrow\!i\in \mathcal{F}$. Hence $h^{I,i}_{s}(x) = x_{i} =  \mathrm{pr}^{I,i}_{s}(x)$. Therefore $h^{I,i}\circ \mathrm{in}^{\varprojlim_{\mathbf{I}}\boldsymbol{\mathcal{A}}} = \mathrm{pr}^{I,i}\circ \mathrm{in}^{\varprojlim_{\mathbf{I}}\boldsymbol{\mathcal{A}}} = f^{i}$.

We are now able to assert that $h^{(\mathbf{I},\mathcal{F}),\boldsymbol{\mathcal{A}}}\circ \mathrm{pr}^{\equiv^{\mathcal{F}}}\circ \mathrm{in}^{\varprojlim_{\mathbf{I}}\boldsymbol{\mathcal{A}}} = \mathrm{id}_{\varprojlim_{\mathbf{I}}\boldsymbol{\mathcal{A}}}$, thereby completing the proof.
\end{proof}

\begin{remark}
If, following L. Ribes and P. Zalesskii in~\cite{rz00}, but for many-sorted algebras, one defines a profinite $\Sigma$-algebra as a projective limit of a projective system of finite $\Sigma$-algebras $\boldsymbol{\mathcal{A}}$ relative to a nonempty upward directed \emph{poset} $\mathbf{I}$ such that the transition homomorphisms of $\boldsymbol{\mathcal{A}}$ are \emph{surjective}, then the just proved theorem still holds, since, by Proposition~\ref{propssupport}, the surjectivity of the transition homomorphisms entails that the $I$-indexed family of $\Sigma$-algebras $(\mathbf{A}^{i})_{i\in I}$ is with constant support. This fact, we think, shows the naturalness of the condition imposed on $(\mathbf{A}^{i})_{i\in I}$.
\end{remark}

\section{A category-theoretic view of the many-sorted version of Mariano-Miraglia theorem.}

Our objective in this section is to provide a categorial rendering of the many-sorted version of Mariano-Miraglia theorem stated in the previous section. To that purpose we consider, by means of the Grothendieck construction for a covariant functor $\mathrm{Uffs}$ from the category $\mathbf{UdPros}_{\neq\varnothing,\mathrm{cof}}^{\mathrm{inj}}$, of nonempty upward directed preordered sets and injective, isotone, and cofinal mappings  between them, to the category of sets, the category $\mathbf{Uffs} = \int_{\mathbf{UdPros}_{\neq\varnothing,\mathrm{cof}}^{\mathrm{inj}}}\mathrm{Uffs}$, in which the objects are the pairs formed by a nonempty upward directed preordered set and by an ultrafilter containing the filter of the final sections of it. Specifically, we show that there exists a functor from the category $\mathbf{Uffs}$ whose object mapping assigns to an object of it a natural transformation between two functors from a suitable category of projective systems of $\Sigma$-algebras to the category of $\Sigma$-algebras, which is a retraction. This is precisely the category-theoretic counterpart of the aforementioned theorem.

But before doing that, since it will prove to be necessary later, we next recall that given a mapping $\varphi$ from a nonempty set $I$ to another $P$ and given an ultrafilter $\mathcal{F}$ on $I$ the co-optimal lift of $\varphi\colon (I,\mathcal{F})\mor P$ is an ultrafilter on $P$.

\begin{proposition}\label{co-optimal lift Ulf}
Let $I$ be a nonempty set, $\mathcal{F}$ an ultrafilter on $I$, and $\varphi$ a mapping from $I$ to $P$. Then
$$
\mathcal{F}_{\varphi[\![\mathcal{F}]\!]} = \{Q\subseteq P\mid \exists\,J\in \mathcal{F}\,(\varphi[J]\subseteq Q)\},
$$
the co-optimal lift of $\varphi\colon (I,\mathcal{F})\mor P$, i.e., the filter on $P$ generated by the filter basis $\varphi[\![\mathcal{F}]\!] = \{\varphi[J]\mid J\in \mathcal{F}\}$ on $P$, is an ultrafilter on $P$.
\end{proposition}

We warn the reader that in what follows the assumption at the beginning of the above section remains in force, i.e., we assume that all preordered sets are nonempty and upward directed.

To achieve the previously mentioned objective we start by defining a convenient category, $\mathbf{UdPros}_{\neq\varnothing,\mathrm{cof}}^{\mathrm{inj}}$, and then a suitable functor, $\mathrm{Uffs}$, from it to $\mathbf{Set}$ from which, by means of the Grothendieck construction, we will obtain the category, $\int_{\mathbf{UdPros}_{\neq\varnothing,\mathrm{cof}}^{\mathrm{inj}}}\mathrm{Uffs}$, which is at the basis of the aforesaid categorial rendering.

\begin{definition}
We denote by $\mathbf{UdPros}_{\neq\varnothing,\mathrm{cof}}^{\mathrm{inj}}$ the category whose objects are the preordered sets $\mathbf{I}$ and whose morphisms from $\mathbf{I}$ to $\mathbf{P}$ are the injective, isotone, and cofinal mappings $\varphi$ from $\mathbf{I}$ to $\mathbf{P}$ (recall that $\varphi$ is cofinal if for every $p\in P$ there exists an $i\in I$ such that $p\leq \varphi(i)$).
\end{definition}

\begin{proposition}
There exists a functor $\mathrm{Uffs}$ from $\mathbf{UdPros}_{\neq\varnothing,\mathrm{cof}}^{\mathrm{inj}}$ to $\mathbf{Set}$ which sends $\mathbf{I}$ to $\mathrm{Uffs}(\mathbf{I}) = \{\mathcal{F}\in \mathrm{Ufilt}(I)\mid \{\Uparrow\!i\mid i\in I\}\subseteq \mathcal{F}\}$, where $\mathrm{Ufilt}(I)$ is the set of all ultrafilters on $I$, and $\varphi\colon\mathbf{I}\mor\mathbf{P}$ to the mapping $\mathrm{Uffs}(\varphi)$ from $\mathrm{Uffs}(\mathbf{I})$ to $\mathrm{Uffs}(\mathbf{P})$ that assigns to each  $\mathcal{F}$ in $\mathrm{Uffs}(\mathbf{I})$ precisely $\mathcal{F}_{\varphi[\![\mathcal{F}]\!]}$ in $\mathrm{Uffs}(\mathbf{P})$.
\end{proposition}

\begin{proof}
We begin by proving that $\mathrm{Uffs}(\varphi)$ is well defined. This is so because, on the one hand, by Proposition~\ref{co-optimal lift Ulf}, $\mathcal{F}_{\varphi[\![\mathcal{F}]\!]}$ is an ultrafilter on $P$ and, on the other hand, since $\varphi$ is isotone and cofinal, the filter basis $\{\Uparrow\!p\mid p\in P\}$ is included in $\mathcal{F}_{\varphi[\![\mathcal{F}]\!]}$.

Since, evidently, $\mathrm{Uffs}$ preserves identities, let us show that if $\psi$ a morphism from $\mathbf{P}$ to $\mathbf{W}$, then $\mathrm{Uffs}(\psi\circ\varphi) = \mathrm{Uffs}(\psi)\circ\mathrm{Uffs}(\varphi)$, i.e., for every $\mathcal{F}\in\mathrm{Uffs}(\mathbf{I})$, we have that $\mathcal{F}_{(\psi\circ\varphi)[\![\mathcal{F}]\!]} = \mathcal{F}_{\psi[\![\mathcal{F}_{\varphi[\![\mathcal{F}]\!]}]\!]}$. Let $\mathcal{F}$ be an element of $\mathrm{Uffs}(\mathbf{I})$ and $X\subseteq W$ an element of $\mathcal{F}_{(\psi\circ\varphi)[\![\mathcal{F}]\!]}$. Then there exists a $J\in \mathcal{F}$ such that $\psi[\varphi[J]]\subseteq X$. Therefore, for $Q = \varphi[J]\in \mathcal{F}_{\varphi[\![\mathcal{F}]\!]}$, we have that $\psi[Q]\subseteq X$. Hence $X\in\mathcal{F}_{\psi[\![\mathcal{F}_{\varphi[\![\mathcal{F}]\!]}]\!]}$. Thus $\mathcal{F}_{(\psi\circ\varphi)[\![\mathcal{F}]\!]} \subseteq\mathcal{F}_{\psi[\![\mathcal{F}_{\varphi[\![\mathcal{F}]\!]}]\!]}$. But $\mathcal{F}_{(\psi\circ\varphi)[\![\mathcal{F}]\!]}$ is an ultrafilter on $W$, consequently, $\mathcal{F}_{(\psi\circ\varphi)[\![\mathcal{F}]\!]} = \mathcal{F}_{\psi[\![\mathcal{F}_{\varphi[\![\mathcal{F}]\!]}]\!]}$.
\end{proof}

\begin{definition}
We denote by $\mathbf{Uffs}$ the category $\int_{\mathbf{UdPros}_{\neq\varnothing,\mathrm{cof}}^{\mathrm{inj}}}\mathrm{Uffs}$ (obtained by means of the Grothendieck construction for the covariant functor $\mathrm{Uffs}$) whose objects are the ordered pairs $(\mathbf{I},\mathcal{F}_{\mathbf{I}})$ where $\mathbf{I}$ is an object of $\mathbf{UdPros}_{\neq\varnothing,\mathrm{cof}}^{\mathrm{inj}}$ and $\mathcal{F}_{\mathbf{I}}\in \mathrm{Uffs}(\mathbf{I})$, i.e., an ultrafilter on $I$ such that the filter of the final sections of $\mathbf{I}$ is contained in $\mathcal{F}_{\mathbf{I}}$, and whose morphisms from $(\mathbf{I},\mathcal{F}_{\mathbf{I}})$ to $(\mathbf{P},\mathcal{F}_{\mathbf{P}})$ are the injective, isotone, and cofinal mappings $\varphi$ from $\mathbf{I}$ to $\mathbf{P}$ such that $\mathcal{F}_{\varphi[\![\mathcal{F}_{\mathbf{I}}]\!]} = \mathcal{F}_{\mathbf{P}}$.
\end{definition}

\begin{proposition}\label{NatTrans}
Let $(\mathbf{I},\mathcal{F}_{\mathbf{I}})$ be an object of the category $\mathbf{Uffs}$. Then we have the functor $\varprojlim_{\mathbf{I}}\colon \mathbf{Alg}(\Sigma)^{\mathbf{I}^{\mathrm{op}}}\mor \mathbf{Alg}(\Sigma)$ which sends a projective system $\boldsymbol{\mathcal{A}} = ((\mathbf{A}^{i})_{i\in I},(f^{j,i})_{(i,j)\in \leq})$ relative to $\mathbf{I}$ to $\varprojlim_{\mathbf{I}}\boldsymbol{\mathcal{A}}$ and a morphism $u = (u^{i})_{i\in I}$ from $\boldsymbol{\mathcal{A}}$ to $\boldsymbol{\mathcal{B}} = ((\mathbf{B}^{i})_{i\in I},(g^{j,i})_{(i,j)\in \leq})$ to the homomorphism  $\varprojlim_{\mathbf{I}}u$ from $\varprojlim_{\mathbf{I}}\boldsymbol{\mathcal{A}}$ to $\varprojlim_{\mathbf{I}}\boldsymbol{\mathcal{B}}$. Moreover, we have the functor $D_{(\mathbf{I},\mathcal{F}_{\mathbf{I}})}\colon \mathbf{Alg}(\Sigma)^{\mathbf{I}^{\mathrm{op}}}\mor \mathbf{Alg}(\Sigma)^{\boldsymbol{\mathcal{F}_{\mathbf{I}}}}$ which sends a projective system $\boldsymbol{\mathcal{A}}$ relative to  $\mathbf{I}$ to the inductive system $\boldsymbol{\mathcal{A}}(\mathcal{F}_{\mathbf{I}})$ relative $\boldsymbol{\mathcal{F}_{\mathbf{I}}}$ and a morphism $u$ from $\boldsymbol{\mathcal{A}}$ to $\boldsymbol{\mathcal{B}}$ to the morphism $(u(J))_{J\in \mathcal{F}_{\mathbf{I}}}$ from $\boldsymbol{\mathcal{A}}(\mathcal{F}_{\mathbf{I}})$ to $\boldsymbol{\mathcal{B}}(\mathcal{F}_{\mathbf{I}})$. In addition, we have the functor $\varinjlim_{\boldsymbol{\mathcal{F}_{\mathbf{I}}}}\colon \mathbf{Alg}(\Sigma)^{\boldsymbol{\mathcal{F}_{\mathbf{I}}}}\mor \mathbf{Alg}(\Sigma)$. Therefore, we have the functors $\varprojlim_{\mathbf{I}}$ and $\varinjlim_{\boldsymbol{\mathcal{F}_{\mathbf{I}}}}\circ D_{(\mathbf{I},\mathcal{F}_{\mathbf{I}})}$ from $\mathbf{Alg}(\Sigma)^{\mathbf{I}^{\mathrm{op}}}$ to $\mathbf{Alg}(\Sigma)$. If we denote by $\mathbf{Alg}(\Sigma)^{\mathbf{I}^{\mathrm{op}}}_{\mathrm{f,cs}}$ the full subcategory of $\mathbf{Alg}(\Sigma)^{\mathbf{I}^{\mathrm{op}}}$ determined by the projective systems $\boldsymbol{\mathcal{A}}$ relative to $\mathbf{I}$ such that $(\mathbf{A}^{i})_{i\in I}$ is with constant support and, for every $i\in I$, $\mathbf{A}^{i}$ is finite, and, for simplicity of notation, we let $\varprojlim_{\mathbf{I}}$ and $\varinjlim_{\boldsymbol{\mathcal{F}_{\mathbf{I}}}}\circ D_{(\mathbf{I},\mathcal{F}_{\mathbf{I}})}$ stand for the restrictions to $\mathbf{Alg}(\Sigma)^{\mathbf{I}^{\mathrm{op}}}_{\mathrm{f,cs}}$ of the previous functors $\varprojlim_{\mathbf{I}}$ and $\varinjlim_{\boldsymbol{\mathcal{F}_{\mathbf{I}}}}\circ D_{(\mathbf{I},\mathcal{F}_{\mathbf{I}})}$, then it happens that $h^{(\mathbf{I},\mathcal{F}_{\mathbf{I}}),\boldsymbol{\cdot}} = (h^{(\mathbf{I},\mathcal{F}_{\mathbf{I}}),\boldsymbol{\mathcal{A}}})_{\boldsymbol{\mathcal{A}}\in \mathrm{Ob}(\mathbf{Alg}(\Sigma)^{\mathbf{I}^{\mathrm{op}}}_{\mathrm{f,cs}})}$ is a natural transformation from $\varinjlim_{\boldsymbol{\mathcal{F}_{\mathbf{I}}}}\circ D_{(\mathbf{I},\mathcal{F}_{\mathbf{I}})}$ to $\varprojlim_{\mathbf{I}}$, i.e., for every morphism  $u$ from $\boldsymbol{\mathcal{A}}$ to $\boldsymbol{\mathcal{B}}$, the following diagram commutes:
$$\xymatrix@C=40pt@R=40pt{
   \varinjlim_{\boldsymbol{\mathcal{F}_{\mathbf{I}}}}\boldsymbol{\mathcal{A}}(\mathcal{F}_{\mathbf{I}})
   \ar[d]_{\varinjlim_{\boldsymbol{\mathcal{F}_{\mathbf{I}}}} (u(J))_{J\in \mathcal{F}_{\mathbf{I}}}}
     \ar[r]^-{h^{(\mathbf{I},\mathcal{F}_{\mathbf{I}}),\boldsymbol{\mathcal{A}}}} & \varprojlim_{\mathbf{I}}\boldsymbol{\mathcal{A}}
     \ar[d]^{\varprojlim_{\mathbf{I}}u}  \\
   \varinjlim_{\boldsymbol{\mathcal{F}_{\mathbf{I}}}}\boldsymbol{\mathcal{B}}(\mathcal{F}_{\mathbf{I}})
   \ar[r]_-{h^{(\mathbf{I},\mathcal{F}_{\mathbf{I}}),\boldsymbol{\mathcal{B}}}} & \varprojlim_{\mathbf{I}}\boldsymbol{\mathcal{B}}
 }
$$
Moreover, we have that $(\mathrm{p}^{I}\circ \mathrm{in}^{\varprojlim_{\mathbf{I}}\boldsymbol{\mathcal{A}}})_{\boldsymbol{\mathcal{A}}\in \mathrm{Ob}(\mathbf{Alg}(\Sigma)^{\mathbf{I}^{\mathrm{op}}}_{\mathrm{f,cs}})}$ is a natural transformation from  $\varprojlim_{\mathbf{I}}$ to $\varinjlim_{\boldsymbol{\mathcal{F}_{\mathbf{I}}}}\circ D_{(\mathbf{I},\mathcal{F}_{\mathbf{I}})}$ and a right inverse for $h^{(\mathbf{I},\mathcal{F}_{\mathbf{I}}),\boldsymbol{\cdot}}$, i.e.,
$$
h^{(\mathbf{I},\mathcal{F}_{\mathbf{I}}),\boldsymbol{\cdot}}\circ (\mathrm{p}^{I}\circ \mathrm{in}^{\varprojlim_{\mathbf{I}}\boldsymbol{\mathcal{A}}})_{\boldsymbol{\mathcal{A}}\in \mathrm{Ob}(\mathbf{Alg}(\Sigma)^{\mathbf{I}^{\mathrm{op}}}_{\mathrm{f,cs}})} = \mathrm{id}_{\varprojlim_{\mathbf{I}}},
$$
where $\mathrm{id}_{\varprojlim_{\mathbf{I}}}$ is the identity natural transformation at the functor $\varprojlim_{\mathbf{I}}$.
\end{proposition}

\begin{proof}
We restrict ourselves to show that $h^{(\mathbf{I},\mathcal{F}_{\mathbf{I}}),\boldsymbol{\cdot}}$ is a natural transformation from $\varinjlim_{\boldsymbol{\mathcal{F}_{\mathbf{I}}}}\circ D_{(\mathbf{I},\mathcal{F}_{\mathbf{I}})}$ to $\varprojlim_{\mathbf{I}}$. Let $u = (u^{i})_{i\in I}$ be a morphism from $\boldsymbol{\mathcal{A}}$ to $\boldsymbol{\mathcal{B}}$. We claim that  $\varprojlim_{\mathbf{I}} u\circ h^{(\mathbf{I},\mathcal{F}_{\mathbf{I}}),\boldsymbol{\mathcal{A}}} = h^{(\mathbf{I},\mathcal{F}_{\mathbf{I}}),\boldsymbol{\mathcal{B}}}\circ \varinjlim_{\boldsymbol{\mathcal{F}_{\mathbf{I}}}} (u(J))_{J\in \mathcal{F}_{\mathbf{I}}}$. Indeed, this follows from the following facts: (1) $\varinjlim_{\boldsymbol{\mathcal{F}_{\mathbf{I}}}}\boldsymbol{\mathcal{A}}(\mathcal{F}_{\mathbf{I}})$ is an (extremal epi)-sink, (2) $\varprojlim_{\mathbf{I}}\boldsymbol{\mathcal{B}}$ is an (extremal mono)-source, and (3), for every $J\in \mathcal{F}_{\mathbf{I}}$ and every $i\in I$, the homomorphisms $u^{i}\circ h^{J,i}$ and $h^{J,i}\circ u(J)$ from $\mathbf{A}(J)$ to $\mathbf{B}^{i}$ are equal, where, by abuse of notation, we have used the same symbol $h^{J,i}$ for the homomorphisms from $\mathbf{A}(J)$ to $\mathbf{A}^{i}$ and from $\mathbf{B}(J)$ to $\mathbf{B}^{i}$. With regard to the last fact, we recall that, for $s\in \mathrm{supp}_{S}(A^{i})$,  $x\in A(J)_{s}$, and $y\in A^{i}_{s}$, $V^{J,i,s}(x,y) = \{j\in J\cap\Uparrow\!i\mid f^{j,i}_{s}(x_{j}) = y\}$ and that $h^{J,i}_{s}(x) = y$ if, and only if, $V^{J,i,s}(x,y)\in \mathcal{F}_{\mathbf{I}}$. Thus, for $j\in V^{J,i,s}(x,y)$, since, by hypothesis, $u$ is a morphism from $\boldsymbol{\mathcal{A}}$ to $\boldsymbol{\mathcal{B}}$, we have that $g^{j,i}(u^{j}_{s}(x_{j})) = u^{i}_{s}(f^{j,i}_{s}(x_{j})) = u^{i}_{s}(y)$, and so $j\in V^{J,i,s}((u^{j}_{s}(x_{j}))_{j\in J},u^{i}_{s}(y))$. Hence $V^{J,i,s}((u^{j}_{s}(x_{j}))_{j\in J},u^{i}_{s}(y))\in \mathcal{F}_{\mathbf{I}}$, i.e., $h^{J,i}((u^{j}_{s}(x_{j}))_{j\in J}) = u^{i}_{s}(y)$. Therefore $h^{J,i}\circ u(J) = u^{i}\circ h^{J,i}$.
\end{proof}

\begin{conventions}
In what follows, for simplicity of notation, given a functor $F$ from $\mathbf{A}$ to $\mathbf{B}$ and a natural transformation $\eta$ from $G$ to $H$, where $G$ and $H$ are functors from $\mathbf{B}$ to $\mathbf{C}$, $\eta\ast F$ stands for $\eta\ast \mathrm{id}_{F}$, the horizontal composition of $\mathrm{id}_{F}$ and $\eta$, where $\mathrm{id}_{F}$ is the identity natural transformation at $F$, and we write $F\circ F$ for $ \mathrm{id}_{F}\circ  \mathrm{id}_{F}$, the vertical composition of $\mathrm{id}_{F}$ with itself. Moreover, if $\mathbf{X}$ and $\mathbf{Y}$ are subcategories of $\mathbf{A}$ and $\mathbf{B}$, respectively, and there exists the bi-restriction of $F$ to $\mathbf{X}$ and $\mathbf{Y}$, then we denote it briefly by $F$.
\end{conventions}

\begin{proposition}\label{NatTrans p and q}
Let $\varphi\colon (\mathbf{I},\mathcal{F}_{\mathbf{I}}) \mor(\mathbf{P},\mathcal{F}_{\mathbf{P}})$ be a morphism in $\mathbf{Uffs}$. Then $\varphi$ determines a functor $\mathrm{Alg}(\Sigma)^{\varphi}\colon \mathbf{Alg}(\Sigma)^{\mathbf{P}^{\mathrm{op}}}\mor \mathbf{Alg}(\Sigma)^{\mathbf{I}^{\mathrm{op}}}$ which assigns to a projective system $\boldsymbol{\mathcal{A}} = ((\mathbf{A}^{p})_{p\in P},(f^{q,p})_{(p,q)\in \leq})$ in $\mathbf{Alg}(\Sigma)^{\mathbf{P}^{\mathrm{op}}}$ the projective system $\boldsymbol{\mathcal{A}}^{\varphi} = ((\mathbf{A}^{\varphi(i)})_{i\in I},(f^{\varphi(j),\varphi(i)})_{(i,j)\in \leq})$ in $\mathbf{Alg}(\Sigma)^{\mathbf{I}^{\mathrm{op}}}$ and to a morphism $u$ from $\boldsymbol{\mathcal{A}}$ to  $\boldsymbol{\mathcal{B}}$ in $\mathbf{Alg}(\Sigma)^{\mathbf{P}^{\mathrm{op}}}$ the morphism $u^{\varphi} = (u^{\varphi(i)})_{i\in I}$ from $\boldsymbol{\mathcal{A}}^{\varphi}$ to $\boldsymbol{\mathcal{B}}^{\varphi}$ in $\mathbf{Alg}(\Sigma)^{\mathbf{I}^{\mathrm{op}}}$. Therefore, for the categories   $\mathbf{Alg}(\Sigma)^{\mathbf{P}^{\mathrm{op}}}_{\mathrm{f,cs}}$ and $\mathbf{Alg}(\Sigma)^{\mathbf{I}^{\mathrm{op}}}_{\mathrm{f,cs}}$, since there exists the bi-restriction of the functor $\mathrm{Alg}(\Sigma)^{\varphi}$ to them and, by Proposition~\ref{NatTrans}, $h^{(\mathbf{I},\mathcal{F}_{\mathbf{I}}),\boldsymbol{\cdot}}$ is a natural transformation from $\varinjlim_{\boldsymbol{\mathcal{F}_{\mathbf{I}}}}\circ D_{(\mathbf{I},\mathcal{F}_{\mathbf{I}})}$ to $\varprojlim_{\mathbf{I}}$, we have a natural transformation $h^{(\mathbf{I},\mathcal{F}_{\mathbf{I}}),\boldsymbol{\cdot}}\ast \mathrm{Alg}^{\varphi} (= h^{(\mathbf{I},\mathcal{F}_{\mathbf{I}}),\boldsymbol{\cdot}}\ast \mathrm{id}_{\mathrm{Alg}^{\varphi}})$ from $\varinjlim_{\boldsymbol{\mathcal{F}_{\mathbf{I}}}}\circ D_{(\mathbf{I},\mathcal{F}_{\mathbf{I}})}\circ \mathrm{Alg}(\Sigma)^{\varphi}$ to $\varprojlim_{\mathbf{I}}\circ \mathrm{Alg}(\Sigma)^{\varphi}$. Moreover, there exists a natural transformation $\mathfrak{p}^{\varphi}$ from $\varprojlim_{\mathbf{P}}$ to $\varprojlim_{\mathbf{I}}\circ \mathrm{Alg}(\Sigma)^{\varphi}$. On the other hand, also by Proposition~\ref{NatTrans}, for $\mathbf{Alg}(\Sigma)^{\mathbf{P}^{\mathrm{op}}}_{\mathrm{f,cs}}$, we have a natural transformation $h^{(\mathbf{P},\mathcal{F}_{\mathbf{P}}),\boldsymbol{\cdot}}$ from $\varinjlim_{\boldsymbol{\mathcal{F}_{\mathbf{P}}}}\circ D_{(\mathbf{P},\mathcal{F}_{\mathbf{P}})}$ to $\varprojlim_{\mathbf{P}}$. Besides, there exists a natural transformation $\mathfrak{q}^{\varphi}$ from $\varinjlim_{\boldsymbol{\mathcal{F}_{\mathbf{I}}}}\circ D_{(\mathbf{I},\mathcal{F}_{\mathbf{I}})}\circ \mathrm{Alg}(\Sigma)^{\varphi}$ to $\varinjlim_{\boldsymbol{\mathcal{F}_{\mathbf{P}}}}\circ D_{(\mathbf{P},\mathcal{F}_{\mathbf{P}})}$.
\end{proposition}

\begin{proof}
Let $\boldsymbol{\mathcal{A}}$ be a projective system in $\mathbf{Alg}(\Sigma)^{\mathbf{P}^{\mathrm{op}}}_{\mathrm{f,cs}}$. Then, by the universal property of $\varprojlim_{\mathbf{I}}\boldsymbol{\mathcal{A}}^{\varphi}$, since, for every $(i,j)\in \leq$,   $f^{\varphi(i)} = f^{\varphi(j),\varphi(i)}\circ f^{\varphi(j)}$, there exists a unique homomorphism $\mathfrak{p}^{\varphi}_{\boldsymbol{\mathcal{A}}}$ from $\varprojlim_{\mathbf{P}}\boldsymbol{\mathcal{A}}$ to $\varprojlim_{\mathbf{I}}\boldsymbol{\mathcal{A}}^{\varphi}$ such that, for every $i\in I$, $f^{\varphi,i}\circ \mathfrak{p}^{\varphi}_{\boldsymbol{\mathcal{A}}} = f^{\varphi(i)}$, where $f^{\varphi,i}$ is the canonical homomorphism from $\varprojlim_{\mathbf{I}}\boldsymbol{\mathcal{A}}^{\varphi}$ to $\mathbf{A}^{\varphi(i)}$, and then $\mathfrak{p}^{\varphi} = (\mathfrak{p}^{\varphi}_{\boldsymbol{\mathcal{A}}})_{\boldsymbol{\mathcal{A}}\in \mathrm{Ob}(\mathbf{Alg}(\Sigma)^{\mathbf{P}^{\mathrm{op}}}_{\mathrm{f,cs}})}$ is, obviously, a natural transformation from $\varprojlim_{\mathbf{P}}$ to $\varprojlim_{\mathbf{I}}\circ \mathrm{Alg}(\Sigma)^{\varphi}$.

By a similar argument, but for inductive limits, it follows the existence of $\mathfrak{q}^{\varphi}$.
\end{proof}

\begin{proposition}\label{Cylinder equation}
Let $\varphi\colon (\mathbf{I},\mathcal{F}_{\mathbf{I}}) \mor(\mathbf{P},\mathcal{F}_{\mathbf{P}})$ be a morphism in  $\mathbf{Uffs}$. Then, by restricting to $\mathbf{Alg}(\Sigma)^{\mathbf{P}^{\mathrm{op}}}_{\mathrm{f,cs}}$ and $\mathbf{Alg}(\Sigma)^{\mathbf{I}^{\mathrm{op}}}_{\mathrm{f,cs}}$, we have that
$$
h^{(\mathbf{I},\mathcal{F}_{\mathbf{I}}),\boldsymbol{\cdot}}\ast \mathrm{Alg}^{\varphi} = \mathfrak{p}^{\varphi}\circ h^{(\mathbf{P},\mathcal{F}_{\mathbf{P}}),\boldsymbol{\cdot}}\circ \mathfrak{q}^{\varphi},
$$
i.e., in the following diagram the involved natural transformations satisfy the just stated equation:
$$
\xymatrix@C=86pt@R=38pt{
\mathbf{Alg}(\Sigma)
  \ar[rd]^{\mathrm{Id}_{\mathbf{Alg}(\Sigma)}}="T"
  \ar@/^60pt/[dd];[]|(0.50)*+[l]{\varinjlim_{\boldsymbol{\mathcal{F}_{\mathbf{I}}}}\circ D_{(\mathbf{I},\mathcal{F}_{\mathbf{I}})}\circ \mathrm{Alg}(\Sigma)^{\varphi}}="Kl"
  \ar@/_40pt/[dd];[]|*+[r]{\varprojlim_{\mathbf{I}}\circ\mathrm{Alg}(\Sigma)^{\varphi}}="Kl'"
  \\
&
\mathbf{Alg}(\Sigma)
  \ar@/^44pt/[dd];[]|(0.50)*+[l]{\varinjlim_{\boldsymbol{\mathcal{F}_{\mathbf{P}}}}\circ D_{(\mathbf{P},\mathcal{F}_{\mathbf{P}})}}="Kr"
  \ar@/_40pt/[dd];[]|*+[r]{\varprojlim_{\mathbf{P}}}="Kr'"
  \\
\mathbf{Alg}(\Sigma)^{\mathbf{P}^{\mathrm{op}}}_{\mathrm{f,cs}}
  \ar[rd]_{\mathrm{Id}_{\mathbf{Alg}(\Sigma)^{\mathbf{P}^{\mathrm{op}}}_{\mathrm{f,cs}}}}="T'"
  \\
&
\mathbf{Alg}(\Sigma)^{\mathbf{P}^{\mathrm{op}}}_{\mathrm{f,cs}}
\ar @{} "Kl";"Kl'"|{\dir{==>}}^*+{\hspace{0.6cm}h^{(\mathbf{I},\mathcal{F}_{\mathbf{I}}),\boldsymbol{\cdot}}\ast\mathrm{Alg}(\Sigma)^{\varphi}}
\ar @{} "Kr";"Kr'"|{\dir{==>}}^*+{h^{(\mathbf{P},\mathcal{F}_{\mathbf{P}}),\boldsymbol{\cdot}}}
\ar @{} "Kl";"Kr"|*:a(-0)@_{==>}^*+{\mathfrak{q}^{\varphi}}
\ar @{} "Kl'";"Kr'"|*:a(-180)@_{==>}^*+{\hspace{0.2cm}{\mathfrak{p}^{\varphi}}}
}
$$
\end{proposition}

\begin{proof}
Let $\boldsymbol{\mathcal{A}}$ be a projective system in $\mathbf{Alg}(\Sigma)^{\mathbf{P}^{\mathrm{op}}}_{\mathrm{f,cs}}$. We want to show that the homomorphisms $h^{(\mathbf{I},\mathcal{F}_{\mathbf{I}}),\boldsymbol{\mathcal{A}}^{\varphi}}$ and  $\mathfrak{p}^{\varphi}_{\boldsymbol{\mathcal{A}}}\circ h^{(\mathbf{P},\mathcal{F}_{\mathbf{P}}),\boldsymbol{\mathcal{A}}}\circ \mathfrak{q}^{\varphi}_{\boldsymbol{\mathcal{A}}}$ from $\varinjlim_{\boldsymbol{\mathcal{F}_{\mathbf{I}}}}\boldsymbol{\mathcal{A}}^{\varphi}(\mathcal{F}_{\mathbf{I}})$ to $\varprojlim_{\mathbf{I}}\boldsymbol{\mathcal{A}}^{\varphi}$ are identical. To this end, taking into account that $\varprojlim_{\mathbf{I}}\boldsymbol{\mathcal{A}}^{\varphi}$ is an (extremal mono)-source, it suffices to verify that, for every $i\in I$, $f^{\varphi,i}\circ h^{(\mathbf{I},\mathcal{F}_{\mathbf{I}}),\boldsymbol{\mathcal{A}}^{\varphi}}$ is identical to $f^{\varphi,i}\circ \mathfrak{p}^{\varphi}_{\boldsymbol{\mathcal{A}}}\circ h^{(\mathbf{P},\mathcal{F}_{\mathbf{P}}),\boldsymbol{\mathcal{A}}}\circ \mathfrak{q}^{\varphi}_{\boldsymbol{\mathcal{A}}}$, where $f^{\varphi,i}$ is the canonical homomorphism from $\varprojlim_{\mathbf{I}}\boldsymbol{\mathcal{A}}^{\varphi}$ to $\mathbf{A}^{\varphi(i)}$. Moreover, one should bear in mind that, since $\varphi$ is, in particular injective, for every $J\in \mathcal{F}_{\mathbf{I}}$, the $\Sigma$-algebras $\prod_{j\in J}\mathbf{A}^{\varphi(j)}$ and $\prod_{\varphi(j)\in \varphi[J]}\mathbf{A}^{\varphi(j)}$ are isomorphic. We know that $f^{\varphi,i}\circ h^{(\mathbf{I},\mathcal{F}_{\mathbf{I}}),\boldsymbol{\mathcal{A}}^{\varphi}} = h^{\boldsymbol{\mathcal{A}}^{\varphi},\varphi(i)}$, where $h^{\boldsymbol{\mathcal{A}}^{\varphi},\varphi(i)}$ is the unique homomorphism from $\varinjlim_{\boldsymbol{\mathcal{F}_{\mathbf{I}}}}\boldsymbol{\mathcal{A}}^{\varphi}(\mathcal{F}_{\mathbf{I}})$ to $\mathbf{A}^{\varphi(i)}$ such that, for every $J\in \mathcal{F}_{\mathbf{I}}$, $h^{\boldsymbol{\mathcal{A}}^{\varphi},\varphi(i)}\circ \mathrm{p}^{J} = h^{\varphi[J],\varphi(i)}$. On the other hand, by definition of $\mathfrak{p}^{\varphi}_{\boldsymbol{\mathcal{A}}}$, we have that   $f^{\varphi,i}\circ \mathfrak{p}^{\varphi}_{\boldsymbol{\mathcal{A}}} = f^{\varphi(i)}$. Moreover, since $h^{(\mathbf{P},\mathcal{F}_{\mathbf{P}}),\boldsymbol{\mathcal{A}}}$ is the unique homomorphism from $\varinjlim_{\boldsymbol{\mathcal{F}_{\mathbf{P}}}}\boldsymbol{\mathcal{A}}(\mathcal{F}_{\mathbf{P}})$ to $\varprojlim_{\mathbf{P}}\boldsymbol{\mathcal{A}}$ such that, for every $p\in P$, $f^{p}\circ h^{(\mathbf{P},\mathcal{F}_{\mathbf{P}}),\boldsymbol{\mathcal{A}}} = h^{\boldsymbol{\mathcal{A}},p}$, where $f^{p}$ is the canonical homomorphism from $\varprojlim_{\mathbf{P}}\boldsymbol{\mathcal{A}}$ to $\mathbf{A}^{p}$, we have that, for every $i\in I$, taking $p = \varphi(i)$, it happens that $f^{\varphi(i)}\circ h^{(\mathbf{P},\mathcal{F}_{\mathbf{P}}),\boldsymbol{\mathcal{A}}} = h^{\boldsymbol{\mathcal{A}},\varphi(i)}$. Now, from $\mathfrak{q}^{\varphi}_{\boldsymbol{\mathcal{A}}}$, which is the unique homomorphism from $\varinjlim_{\boldsymbol{\mathcal{F}_{\mathbf{I}}}}\boldsymbol{\mathcal{A}}^{\varphi}(\mathcal{F}_{\mathbf{I}})$ to $\varinjlim_{\boldsymbol{\mathcal{F}_{\mathbf{P}}}}\boldsymbol{\mathcal{A}}(\mathcal{F}_{\mathbf{P}})$ such that, for every $J\in \mathcal{F}_{\mathbf{I}}$, $\mathfrak{q}^{\varphi}_{\boldsymbol{\mathcal{A}}}\circ \mathrm{p}^{J} = \mathrm{p}^{\varphi[J]}$ (recall that $\prod_{j\in J}\mathbf{A}^{\varphi(j)}\cong\prod_{\varphi(j)\in \varphi[J]}\mathbf{A}^{\varphi(j)}$), we obtain the homomorphism $h^{\boldsymbol{\mathcal{A}},\varphi(i)}\circ \mathfrak{q}^{\varphi}_{\boldsymbol{\mathcal{A}}}$ from $\varinjlim_{\boldsymbol{\mathcal{F}_{\mathbf{I}}}}\boldsymbol{\mathcal{A}}^{\varphi}(\mathcal{F}_{\mathbf{I}})$ to $\mathbf{A}^{\varphi(i)}$. But it also happens that $h^{\boldsymbol{\mathcal{A}}^{\varphi},\varphi(i)}$ is a homomorphism  from $\varinjlim_{\boldsymbol{\mathcal{F}_{\mathbf{I}}}}\boldsymbol{\mathcal{A}}^{\varphi}(\mathcal{F}_{\mathbf{I}})$ to $\mathbf{A}^{\varphi(i)}$. Therefore, since $\varinjlim_{\boldsymbol{\mathcal{F}_{\mathbf{I}}}}\boldsymbol{\mathcal{A}}^{\varphi}(\mathcal{F}_{\mathbf{I}})$ is an (extremal epi)-sink, to show that $h^{\boldsymbol{\mathcal{A}},\varphi(i)}\circ \mathfrak{q}^{\varphi}_{\boldsymbol{\mathcal{A}}} = h^{\boldsymbol{\mathcal{A}}^{\varphi},\varphi(i)}$ it suffices to prove that, for every $J\in \mathcal{F}_{\mathbf{I}}$, the homomorphisms $h^{\boldsymbol{\mathcal{A}},\varphi(i)}\circ \mathfrak{q}^{\varphi}_{\boldsymbol{\mathcal{A}}}\circ \mathrm{p}^{J}$ and $h^{\boldsymbol{\mathcal{A}}^{\varphi},\varphi(i)}\circ \mathrm{p}^{J}$ from $\prod_{j\in J}\mathbf{A}^{\varphi(j)}\cong\prod_{\varphi(j)\in \varphi[J]}\mathbf{A}^{\varphi(i)}$ to $\mathbf{A}^{\varphi(i)}$ are equal. But both homomorphism are identical to $h^{\varphi[J],\varphi(i)}$. Therefore $h^{\boldsymbol{\mathcal{A}},\varphi(i)}\circ \mathfrak{q}^{\varphi}_{\boldsymbol{\mathcal{A}}} = h^{\boldsymbol{\mathcal{A}}^{\varphi},\varphi(i)}$.
\end{proof}

\begin{proposition}
Let $\varphi\colon (\mathbf{I},\mathcal{F}_{\mathbf{I}}) \mor(\mathbf{P},\mathcal{F}_{\mathbf{P}})$ and $\psi\colon (\mathbf{P},\mathcal{F}_{\mathbf{P}}) \mor(\mathbf{W},\mathcal{F}_{\mathbf{W}})$ be two morphisms in $\mathbf{Uffs}$. Then, from the functors $\mathrm{Alg}(\Sigma)^{\varphi}$ and $\mathrm{Alg}(\Sigma)^{\psi}$, we obtain the functor:
$$
\mathrm{Alg}(\Sigma)^{\psi\circ \varphi} = \mathrm{Alg}(\Sigma)^{\varphi}\circ \mathrm{Alg}(\Sigma)^{\psi}\colon \mathbf{Alg}(\Sigma)^{\mathbf{W}^{\mathrm{op}}}\mor\mathbf{Alg}(\Sigma)^{\mathbf{I}^{\mathrm{op}}}.
$$
Moreover, we have the following natural transformations:
\begin{enumerate}
\item $\mathfrak{p}^{\varphi}\colon\varprojlim_{\mathbf{P}}\cel\varprojlim_{\mathbf{I}}\circ \mathrm{Alg}(\Sigma)^{\varphi}$,
\item $\mathfrak{q}^{\varphi}\colon\varinjlim_{\boldsymbol{\mathcal{F}_{\mathbf{I}}}}\circ D_{(\mathbf{I},\mathcal{F}_{\mathbf{I}})}\circ
    \mathrm{Alg}(\Sigma)^{\varphi}\cel\varinjlim_{\boldsymbol{\mathcal{F}_{\mathbf{P}}}}\circ D_{(\mathbf{P},\mathcal{F}_{\mathbf{P}})}$,
\item $\mathfrak{p}^{\psi}\colon\varprojlim_{\mathbf{W}}\cel\varprojlim_{\mathbf{P}}\circ \mathrm{Alg}(\Sigma)^{\psi}$,
\item $\mathfrak{q}^{\psi}\colon\varinjlim_{\boldsymbol{\mathcal{F}_{\mathbf{P}}}}\circ D_{(\mathbf{P},\mathcal{F}_{\mathbf{P}})}\circ \mathrm{Alg}(\Sigma)^{\psi}\cel \varinjlim_{\boldsymbol{\mathcal{F}_{\mathbf{W}}}}\circ D_{(\mathbf{W},\mathcal{F}_{\mathbf{W}})}$,
\item $\mathfrak{p}^{\psi\circ \varphi}\colon\varprojlim_{\mathbf{W}}\cel\varprojlim_{\mathbf{I}}\circ \mathrm{Alg}(\Sigma)^{\psi\circ \varphi}$,
\item $\mathfrak{q}^{\psi\circ \varphi}\colon\varinjlim_{\boldsymbol{\mathcal{F}_{\mathbf{I}}}}\circ D_{(\mathbf{I},\mathcal{F}_{\mathbf{I}})}\circ \mathrm{Alg}(\Sigma)^{\psi\circ \varphi}\cel \varinjlim_{\boldsymbol{\mathcal{F}_{\mathbf{W}}}}\circ D_{(\mathbf{W},\mathcal{F}_{\mathbf{W}})}$,
\item $h^{(\mathbf{I},\mathcal{F}_{\mathbf{I}}),\boldsymbol{\cdot}}\ast \mathrm{Alg}(\Sigma)^{\varphi}\colon \varinjlim_{\boldsymbol{\mathcal{F}_{\mathbf{I}}}}\circ D_{(\mathbf{I},\mathcal{F}_{\mathbf{I}})}\circ \mathrm{Alg}(\Sigma)^{\varphi}\cel\varprojlim_{\mathbf{I}}\circ \mathrm{Alg}(\Sigma)^{\varphi}$,
\item $h^{(\mathbf{P},\mathcal{F}_{\mathbf{P}}),\boldsymbol{\cdot}}\ast \mathrm{Alg}(\Sigma)^{\psi}\colon \varinjlim_{\boldsymbol{\mathcal{F}_{\mathbf{P}}}}\circ D_{(\mathbf{P},\mathcal{F}_{\mathbf{P}})}\circ \mathrm{Alg}(\Sigma)^{\psi}\cel\varprojlim_{\mathbf{P}}\circ \mathrm{Alg}(\Sigma)^{\psi}$,
\item $h^{(\mathbf{I},\mathcal{F}_{\mathbf{I}}),\boldsymbol{\cdot}}\ast \mathrm{Alg}(\Sigma)^{\psi\circ\varphi}\colon \varinjlim_{\boldsymbol{\mathcal{F}_{\mathbf{I}}}}\circ D_{(\mathbf{I},\mathcal{F}_{\mathbf{I}})}\circ \mathrm{Alg}(\Sigma)^{\psi\circ\varphi}\cel\varprojlim_{\mathbf{I}}\circ \mathrm{Alg}(\Sigma)^{\psi\circ\varphi}$,
    \text{and}
\item $h^{(\mathbf{W},\mathcal{F}_{\mathbf{W}}),\boldsymbol{\cdot}}\colon
    \varinjlim_{\boldsymbol{\mathcal{F}_{\mathbf{W}}}}\circ D_{(\mathbf{W},\mathcal{F}_{\mathbf{W}})}\cel\varprojlim_{\mathbf{W}}$.
\end{enumerate}
Then, from $\mathfrak{p}^{\varphi}\colon\varprojlim_{\mathbf{P}}\cel\varprojlim_{\mathbf{I}}\circ \mathrm{Alg}(\Sigma)^{\varphi}$ and the functor $\mathrm{Alg}(\Sigma)^{\psi}$, we obtain the natural transformation:
$$
\textstyle
\mathfrak{p}^{\varphi}\ast \mathrm{Alg}(\Sigma)^{\psi}\colon\varprojlim_{\mathbf{P}}\circ \mathrm{Alg}(\Sigma)^{\psi}\cel\varprojlim_{\mathbf{I}}\circ \mathrm{Alg}(\Sigma)^{\varphi}\circ \mathrm{Alg}(\Sigma)^{\psi},
$$
and, from $\mathfrak{p}^{\varphi}\ast \mathrm{Alg}(\Sigma)^{\psi}$ and $\mathfrak{p}^{\psi}$, we obtain the natural transformation:
$$
\textstyle
(\mathfrak{p}^{\varphi}\ast \mathrm{Alg}(\Sigma)^{\psi})\circ \mathfrak{p}^{\psi}\colon \varprojlim_{\mathbf{W}}\cel
\varprojlim_{\mathbf{I}}\circ \mathrm{Alg}(\Sigma)^{\varphi}\circ \mathrm{Alg}(\Sigma)^{\psi}.
$$
Similarly, from $\mathfrak{q}^{\varphi}\colon\varinjlim_{\boldsymbol{\mathcal{F}_{\mathbf{I}}}}\circ D_{(\mathbf{I},\mathcal{F}_{\mathbf{I}})}\circ\mathrm{Alg}(\Sigma)^{\varphi}\cel
\varinjlim_{\boldsymbol{\mathcal{F}_{\mathbf{P}}}}\circ D_{(\mathbf{P},\mathcal{F}_{\mathbf{P}})}$ and the functor $\mathrm{Alg}(\Sigma)^{\psi}$, we obtain the natural transformation:
$$
\textstyle
\mathfrak{q}^{\varphi}\ast \mathrm{Alg}(\Sigma)^{\psi}\colon\varinjlim_{\boldsymbol{\mathcal{F}_{\mathbf{I}}}}\circ D_{(\mathbf{I},\mathcal{F}_{\mathbf{I}})}\circ\mathrm{Alg}(\Sigma)^{\varphi}\circ \mathrm{Alg}(\Sigma)^{\psi}\cel\varinjlim_{\boldsymbol{\mathcal{F}_{\mathbf{P}}}}\circ D_{(\mathbf{P},\mathcal{F}_{\mathbf{P}})}\circ \mathrm{Alg}(\Sigma)^{\psi},
$$
and, from $\mathfrak{q}^{\varphi}\ast \mathrm{Alg}(\Sigma)^{\psi}$ and $\mathfrak{q}^{\psi}$, we obtain the natural transformation:
$$
\textstyle
\mathfrak{q}^{\psi}\circ(\mathfrak{q}^{\varphi}\ast \mathrm{Alg}(\Sigma)^{\psi})\colon \varinjlim_{\boldsymbol{\mathcal{F}_{\mathbf{I}}}}\circ D_{(\mathbf{I},\mathcal{F}_{\mathbf{I}})}\circ\mathrm{Alg}(\Sigma)^{\varphi}\circ \mathrm{Alg}(\Sigma)^{\psi} \cel\varinjlim_{\boldsymbol{\mathcal{F}_{\mathbf{W}}}}\circ D_{(\mathbf{W},\mathcal{F}_{\mathbf{W}})}.
$$
Then it happens that $\mathfrak{p}^{\psi\circ\varphi} = (\mathfrak{p}^{\varphi}\ast \mathrm{Alg}(\Sigma)^{\psi})\circ \mathfrak{p}^{\psi}$ and $\mathfrak{q}^{\psi\circ\varphi} = \mathfrak{q}^{\psi}\circ(\mathfrak{q}^{\varphi}\ast \mathrm{Alg}(\Sigma)^{\psi})$. Therefore
$$
h^{(\mathbf{I},\mathcal{F}_{\mathbf{I}}),\boldsymbol{\cdot}}\ast \mathrm{Alg}(\Sigma)^{\psi\circ\varphi} = \mathfrak{p}^{\psi\circ\varphi}\circ h^{(\mathbf{W},\mathcal{F}_{\mathbf{W}}),\boldsymbol{\cdot}}\circ \mathfrak{q}^{\psi\circ\varphi}.
$$
\end{proposition}

\begin{proof}
To show that $\mathfrak{p}^{\psi\circ\varphi} = (\mathfrak{p}^{\varphi}\ast \mathrm{Alg}(\Sigma)^{\psi})\circ \mathfrak{p}^{\psi}$ it suffices to verify that, for every projective system $\boldsymbol{\mathcal{A}}$  in $\mathbf{Alg}(\Sigma)^{\mathbf{W}^{\mathrm{op}}}_{\mathrm{f,cs}}$, the homomorphisms
$$
\textstyle
((\mathfrak{p}^{\varphi}\ast \mathrm{Alg}(\Sigma)^{\psi})\circ \mathfrak{p}^{\psi})_{\boldsymbol{\mathcal{A}}} = \mathfrak{p}^{\varphi}_{\boldsymbol{\mathcal{A}}^{\psi}}\circ\mathfrak{p}^{\psi}_{\boldsymbol{\mathcal{A}}},\,\, \mathfrak{p}^{\psi\circ\varphi}_{\boldsymbol{\mathcal{A}}}\colon \varprojlim_{\mathbf{W}}\boldsymbol{\mathcal{A}}\mor \varprojlim_{\mathbf{I}}\boldsymbol{\mathcal{A}}^{\psi\circ\varphi}
$$
are equal. But it happens that $\mathfrak{p}^{\psi\circ\varphi}_{\boldsymbol{\mathcal{A}}}$ is the unique homomorphism from $\varprojlim_{\mathbf{W}}\boldsymbol{\mathcal{A}}$ to $\varprojlim_{\mathbf{I}}\boldsymbol{\mathcal{A}}^{\psi\circ\varphi}$ such that, for every $i\in I$, $f^{\psi\circ\varphi,i}\circ \mathfrak{p}^{\psi\circ\varphi}_{\boldsymbol{\mathcal{A}}} = f^{\psi(\varphi(i))}$, where $f^{\psi\circ\varphi,i}$ is the canonical homomorphism from $\varprojlim_{\mathbf{I}}\boldsymbol{\mathcal{A}}^{\psi\circ\varphi}$ to $\mathbf{A}^{\psi(\varphi(i))}$, and, for every $i\in I$, we have that
\begin{align}
f^{\psi\circ\varphi,i}\circ
(\mathfrak{p}^{\varphi}_{\boldsymbol{\mathcal{A}}^{\psi}}\circ\mathfrak{p}^{\psi}_{\boldsymbol{\mathcal{A}}}) &=
(f^{\psi\circ\varphi,i}\circ\mathfrak{p}^{\varphi}_{\boldsymbol{\mathcal{A}}^{\psi}})
\circ\mathfrak{p}^{\psi}_{\boldsymbol{\mathcal{A}}} \notag \\
&= f^{\psi,\varphi(i)}\circ \mathfrak{p}^{\psi}_{\boldsymbol{\mathcal{A}}} \notag \\
&= f^{\psi(\varphi(i))}. \notag
\end{align}
Therefore $((\mathfrak{p}^{\varphi}\ast \mathrm{Alg}(\Sigma)^{\psi})\circ \mathfrak{p}^{\psi})_{\boldsymbol{\mathcal{A}}} = \mathfrak{p}^{\psi\circ\varphi}_{\boldsymbol{\mathcal{A}}}$. Hence $\mathfrak{p}^{\psi\circ\varphi} = (\mathfrak{p}^{\varphi}\ast \mathrm{Alg}(\Sigma)^{\psi})\circ \mathfrak{p}^{\psi}$.

By a similar argument it follows that $\mathfrak{q}^{\psi\circ\varphi} = \mathfrak{q}^{\psi}\circ(\mathfrak{q}^{\varphi}\ast \mathrm{Alg}(\Sigma)^{\psi})$.

It remains to show that
$$
h^{(\mathbf{I},\mathcal{F}_{\mathbf{I}}),\boldsymbol{\cdot}}\ast \mathrm{Alg}(\Sigma)^{\psi\circ\varphi} = \mathfrak{p}^{\psi\circ\varphi}\circ h^{(\mathbf{W},\mathcal{F}_{\mathbf{W}}),\boldsymbol{\cdot}}\circ \mathfrak{q}^{\psi\circ\varphi}.
$$
But we have that
\begin{alignat}{2}
\mathfrak{p}^{\psi\circ\varphi}\circ h^{(\mathbf{W},\mathcal{F}_{\mathbf{W}}),\boldsymbol{\cdot}}\circ \mathfrak{q}^{\psi\circ\varphi} &= ((\mathfrak{p}^{\varphi}\ast\mathrm{Alg}(\Sigma)^{\psi})\circ\mathfrak{p}^{\psi})\circ h^{(\mathbf{W},\mathcal{F}_{\mathbf{W}}),\boldsymbol{\cdot}}\circ (\mathfrak{q}^{\psi}\circ(\mathfrak{q}^{\varphi}\ast \mathrm{Alg}(\Sigma)^{\psi})  & & \text{(by def.)}\notag \\
&= (\mathfrak{p}^{\varphi}\ast\mathrm{Alg}(\Sigma)^{\psi})\circ(\mathfrak{p}^{\psi}\circ h^{(\mathbf{W},\mathcal{F}_{\mathbf{W}}),\boldsymbol{\cdot}}\circ \mathfrak{q}^{\psi})\circ(\mathfrak{q}^{\varphi}\ast \mathrm{Alg}(\Sigma)^{\psi})  & & \text{(by ass.)} \notag \\
&= (\mathfrak{p}^{\varphi}\ast\mathrm{Alg}(\Sigma)^{\psi})\circ(h^{(\mathbf{P},\mathcal{F}_{\mathbf{P}}),\boldsymbol{\cdot}}\ast  \mathrm{Alg}(\Sigma)^{\psi})\circ(\mathfrak{q}^{\varphi}\ast \mathrm{Alg}(\Sigma)^{\psi})  & & \text{(by def.)} \notag \\
&= (\mathfrak{p}^{\varphi}\ast\mathrm{Alg}(\Sigma)^{\psi})\circ ((h^{(\mathbf{P},\mathcal{F}_{\mathbf{P}}),\boldsymbol{\cdot}}\circ \mathfrak{q}^{\varphi})\ast \mathrm{Alg}(\Sigma)^{\psi}) & & \hspace{-1.15cm}\text{(Godement law)} \notag \\
&= (\mathfrak{p}^{\varphi}\circ h^{(\mathbf{P},\mathcal{F}_{\mathbf{P}}),\boldsymbol{\cdot}}\circ \mathfrak{q}^{\varphi})\ast \mathrm{Alg}(\Sigma)^{\psi}& & \hspace{-1.15cm}\text{(Godement law)} \notag \\
&= (h^{(\mathbf{I},\mathcal{F}_{\mathbf{I}}),\boldsymbol{\cdot}}\ast\mathrm{Alg}(\Sigma)^{\varphi})\ast\mathrm{Alg}(\Sigma)^{\psi} & & \text{(by def.)}\notag \\
&= h^{(\mathbf{I},\mathcal{F}_{\mathbf{I}}),\boldsymbol{\cdot}}\ast \mathrm{Alg}(\Sigma)^{\psi\circ\varphi} & & \hspace{-1.35cm}\text{(by ass. and def.)}\notag
\end{alignat}

Regarding the natural transformations annotated ``Godement law'', in the equations listed above, one should take into account that for the functor $\mathrm{Alg}(\Sigma)^{\psi}$ from $\mathbf{Alg}(\Sigma)^{\mathbf{W}^{\mathrm{op}}}$ to $\mathbf{Alg}(\Sigma)^{\mathbf{P}^{\mathrm{op}}}$ since, as a particular case of the conventions stated just before Proposition~\ref{NatTrans p and q},
$$
\mathrm{Alg}(\Sigma)^{\psi}\circ \mathrm{Alg}(\Sigma)^{\psi} = \mathrm{id}_{\mathrm{Alg}(\Sigma)^{\psi}}\circ \mathrm{id}_{\mathrm{Alg}(\Sigma)^{\psi}} = \mathrm{id}_{\mathrm{Alg}(\Sigma)^{\psi}} = \mathrm{Alg}(\Sigma)^{\psi},
$$
we have that
\begin{gather}
(h^{(\mathbf{P},\mathcal{F}_{\mathbf{P}}),\boldsymbol{\cdot}}\circ \mathfrak{q}^{\varphi})\ast \mathrm{Alg}(\Sigma)^{\psi} = (h^{(\mathbf{P},\mathcal{F}_{\mathbf{P}}),\boldsymbol{\cdot}}\circ \mathfrak{q}^{\varphi})\ast (\mathrm{Alg}(\Sigma)^{\psi}\circ \mathrm{Alg}(\Sigma)^{\psi}) \text{ and} \notag \\
(\mathfrak{p}^{\varphi}\circ h^{(\mathbf{P},\mathcal{F}_{\mathbf{P}}),\boldsymbol{\cdot}}\circ \mathfrak{q}^{\varphi})\ast \mathrm{Alg}(\Sigma)^{\psi} = (\mathfrak{p}^{\varphi}\circ h^{(\mathbf{P},\mathcal{F}_{\mathbf{P}}),\boldsymbol{\cdot}}\circ \mathfrak{q}^{\varphi})\ast (\mathrm{Alg}(\Sigma)^{\psi}\circ \mathrm{Alg}(\Sigma)^{\psi}).\notag
\end{gather}
\end{proof}

We would like to conclude this article by pointing out that, from the above results and taking into account the work done in~\cite{cs10}, it seems to us that a generalization of the results stated in this section to a 2-categorial setting is feasible.

Let us begin by noticing that the above category-theoretic rendering of Mariano and Miraglia theorem has been done by fixing a pair $\mathbf{\Sigma} = (S,\Sigma)$, where $S$ is a set of sorts and $\Sigma$ an $S$-sorted signature. In making so we have assigned to every object $(\mathbf{I},\mathcal{F}_{\mathbf{I}})$ of $\mathbf{Uffs}$ a natural transformation  $h^{(\mathbf{P},\mathcal{F}_{\mathbf{P}}),\boldsymbol{\cdot}}$, and to every morphism $\varphi$ from $(\mathbf{I},\mathcal{F}_{\mathbf{I}})$ to $(\mathbf{P},\mathcal{F}_{\mathbf{P}})$ in $\mathbf{Uffs}$ a pair of natural transformations $(\mathfrak{p}^{\varphi},\mathfrak{q}^{\varphi})$ satisfying the equation stated in Proposition~\ref{Cylinder equation}. Moreover, we have shown that such a correspondence is, in fact, a functor.

Faced with such a situation, the next, natural, step would be to investigate what happens if one allows the variation of $\mathbf{\Sigma} = (S,\Sigma)$. In this regard, we would note that there exists a contravariant functor $\mathrm{Sig}$ from $\mathbf{Set}$ to $\mathbf{Cat}$. Its object mapping sends each set of sorts $S$ to $\mathrm{Sig}(S) = \mathbf{Sig}(S)$ (= $\mathbf{Set}^{S^{\star}\times S}$), the category of all $S$-sorted signatures; its arrow mapping sends each mapping $\alpha$ from $S$ to $T$ to the functor $\mathrm{Sig}(\alpha)$ from $\mathbf{Sig}(T)$ to $\mathbf{Sig}(S)$ which relabels $T$-sorted signatures into $S$-sorted signatures, i.e., $\mathrm{Sig}(\alpha)$ assigns to a $T$-sorted signature $\Lambda\colon T^{\star}\times T\mor \boldsymbol{\mathcal{U}}$ the $S$-sorted signature $\mathrm{Sig}(\alpha)(\Lambda) = \Lambda_{\alpha^{\star}\times\alpha}$, where $\Lambda_{\alpha^{\star}\times\alpha}$ is the composition of $\alpha^{\star}\times\alpha\colon S^{\star}\times S\mor T^{\star}\times T$ and $\Lambda$, and assigns to a morphism of $T$-sorted signatures $d$ from $\Lambda$ to $\Lambda'$ the morphism of $S$-sorted signatures $\mathrm{Sig}(\alpha)(d) = d_{\alpha^{\star}\times \alpha}$ from $\Lambda_{\alpha^{\star}\times \alpha}$ to $\Lambda'_{\alpha^{\star}\times \alpha}$. Then the category $\mathbf{Sig}$, of \emph{many-sorted signatures} and \emph{many-sorted signature morphisms}, is given by $\mathbf{Sig} = \int^{\mathbf{Set}}\mathrm{Sig}$. Therefore $\mathbf{Sig}$ has as objects the pairs $\mathbf{\Sigma} = (S,\Sigma)$, where $S$ is a set of sorts and $\Sigma$ an $S$-sorted signature, and, as many-sorted signature morphisms from $\mathbf{\Sigma} = (S,\Sigma)$ to $\mathbf{\Lambda} = (T,\Lambda)$, the pairs $\mathbf{d} = (\alpha,d)$, where $\alpha\colon S\mor T$ is a morphism in $\mathbf{Set}$ while $d\colon \Sigma\mor \Lambda_{\alpha^{\star}\times \alpha}$ is a morphism in $\mathbf{Sig}(S)$ (for details see~\cite{cs10}).

Moreover, there exists a contravariant functor $\mathrm{Alg}$ from $\mathbf{Sig}$ to $\mathbf{Cat}$. Its object mapping sends each signature $\mathbf{\Sigma}$ to $\mathrm{Alg}(\mathbf{\Sigma}) = \mathbf{Alg}(\mathbf{\Sigma})$, the category of $\mathbf{\Sigma}$-algebras; its arrow mapping sends each signature morphism $\mathbf{d}\colon\mathbf{\Sigma}\mor\mathbf{\Lambda}$ to the functor $\mathrm{Alg}(\mathbf{d}) = \mathbf{d}^{\ast}\colon \mathbf{Alg}(\mathbf{\Lambda})\mor \mathbf{Alg}(\mathbf{\Sigma})$ defined as follows: its object mapping sends each $\mathbf{\Lambda}$-algebra $\mathbf{B} = (B,G)$ to the $\mathbf{\Sigma}$-algebra $\mathbf{d}^{\ast}(\mathbf{B}) = (B_{\alpha},G^{\mathbf{d}})$, where $B_{\alpha}$ is  $(B_{\alpha(s)})_{s\in S}$ and $G^{\mathbf{d}}$ is the composition of the $S^{\star}\times S$-sorted mappings $d$ from $\Sigma$ to $\Lambda_{\alpha^{\star}\times \alpha}$ and $G_{\alpha^{\star}\times \alpha}$ from $\Lambda_{\alpha^{\star}\times \alpha}$ to $\mathcal{O}_{T}(B)_{\alpha^{\star}\times \alpha}$, where $\mathcal{O}_{T}(B)$ stands for the $T^{\star}\times T$-sorted set $(\mathrm{Hom}(B_{u},B_{t}))_{(u,t)\in T^{\star}\times T}$, of the finitary operations on the $T$-sorted set $B$; its arrow mapping sends each $\mathbf{\Lambda}$-homomorphism $f$ from $\mathbf{B}$ to $\mathbf{B}'$ to the $\mathbf{\Sigma}$-homomorphism $\mathbf{d}^{\ast}(f) = f_{\alpha}$ from $\mathbf{d}^{\ast}(\mathbf{B})$ to $\mathbf{d}^{\ast}(\mathbf{B}')$, where $f_{\alpha}$ is $(f_{\alpha(s)})_{s\in S}$. Then the category $\mathbf{Alg}$, of \emph{many-sorted algebras} and \emph{many-sorted algebra homomorphisms}, is given by $\mathbf{Alg} = \int^{\mathbf{Sig}}\mathrm{Alg}$. Therefore the category $\mathbf{Alg}$ has as objects the pairs $(\mathbf{\Sigma},\mathbf{A})$, where $\mathbf{\Sigma}$ is a signature and $\mathbf{A}$ a $\mathbf{\Sigma}$-algebra, and as morphisms from $(\mathbf{\Sigma},\mathbf{A})$ to $(\mathbf{\Lambda},\mathbf{B})$, the pairs $(\mathbf{d},f)$, with $\mathbf{d}$ a signature morphism from $\mathbf{\Sigma}$ to $\mathbf{\Lambda}$ and $f$ a $\mathbf{\Sigma}$-homomorphism from $\mathbf{A}$ to $\mathbf{d}^{\ast}(\mathbf{B})$ (for details see~\cite{cs10}).

Thus, the new goal would be to assigns to an object $((\mathbf{I},\mathcal{F}_{\mathbf{I}}),\mathbf{\Sigma})$ of the category  $\mathbf{Uffs}\times \mathbf{Sig}$ a natural transformation $h^{((\mathbf{I},\mathcal{F}_{\mathbf{I}}),\mathbf{\Sigma}),\boldsymbol{\cdot}}$ from $\varinjlim_{\boldsymbol{\mathcal{F}_{\mathbf{I}}}}\circ D_{(\mathbf{I},\mathcal{F}_{\mathbf{I}})}$ to $\varprojlim_{\mathbf{I}}$, and to a morphism $(\varphi,\mathbf{d})$ from $((\mathbf{I},\mathcal{F}_{\mathbf{I}}),\mathbf{\Sigma})$ to $((\mathbf{P},\mathcal{F}_{\mathbf{P}}),\mathbf{\Lambda})$ a suitable pair of natural transformations $(\mathfrak{p}^{(\varphi,\mathbf{d})},\mathfrak{q}^{(\varphi,\mathbf{d})})$, where
\begin{gather}
\textstyle
\mathfrak{p}^{(\varphi,\mathbf{d})}\colon \mathbf{d}^{\ast}\ast \varprojlim_{\mathbf{P}}\cel \varprojlim_{\mathbf{I}}\ast((\mathbf{d}^{\ast})^{\mathbf{I}^{\mathrm{op}}}\circ\mathrm{Alg}(\mathbf{\Lambda})^{\varphi})  \text{ and} \notag \\
\textstyle
\mathfrak{q}^{(\varphi,\mathbf{d})}\colon (\varinjlim_{\boldsymbol{\mathcal{F}_{\mathbf{I}}}}\circ D_{(\mathbf{I},\mathcal{F}_{\mathbf{I}})})\ast ((\mathbf{d}^{\ast})^{\mathbf{I}^{\mathrm{op}}}\circ\mathrm{Alg}(\mathbf{\Lambda})^{\varphi})\cel \mathbf{d}^{\ast}\ast \varinjlim_{\boldsymbol{\mathcal{F}_{\mathbf{P}}}}\circ D_{(\mathbf{P},\mathcal{F}_{\mathbf{P}})}.\notag
\end{gather}
To assist the reader in identifying the just stated natural transformations, we add the following diagram:
$$
\xymatrix@C=96pt@R=65pt{
\mathbf{Alg}(\mathbf{\Sigma})^{\mathbf{I}^{\mathrm{op}}}_{\mathrm{f,cs}}
  \ar@/^20pt/[r]^{\varinjlim_{\boldsymbol{\mathcal{F}_{\mathbf{I}}}}\circ D_{(\mathbf{I},\mathcal{F}_{\mathbf{I}})}}="a"
  \ar@/_20pt/[r]_{\varprojlim_{\mathbf{I}}}="b" &
\mathbf{Alg}(\mathbf{\Sigma}) \\
\mathbf{Alg}(\mathbf{\Lambda})^{\mathbf{P}^{\mathrm{op}}}_{\mathrm{f,cs}}
\ar[u]^{(\mathbf{d}^{\ast})^{\mathbf{I}^{\mathrm{op}}}\circ\mathrm{Alg}(\mathbf{\Lambda})^{\varphi}}
  \ar@/^20pt/[r]^{\varinjlim_{\boldsymbol{\mathcal{F}_{\mathbf{P}}}}\circ D_{(\mathbf{P},\mathcal{F}_{\mathbf{P}})}}="c"
  \ar@/_20pt/[r]_{\varprojlim_{\mathbf{P}}}="d" &
\mathbf{Alg}(\mathbf{\Lambda})\ar[u]_{\mathbf{d}^{\ast}}
\ar @{} "a";"b" |{\dir{=>}}^{\,h^{((\mathbf{I},\mathcal{F}_{\mathbf{I}}),\mathbf{\Sigma}),\boldsymbol{\cdot}}}
\ar @{}"c";"d" |{\dir{=>}}^{\,h^{((\mathbf{P},\mathcal{F}_{\mathbf{P}}),\mathbf{\Lambda}),\boldsymbol{\cdot}}}
}
$$

Moreover, since for two morphisms $(\varphi,\mathbf{d}),\,(\varphi',\mathbf{d}')\colon((\mathbf{I},\mathcal{F}_{\mathbf{I}}),\mathbf{\Sigma})\mor
((\mathbf{P},\mathcal{F}_{\mathbf{P}}),\mathbf{\Lambda})$ there exists a natural notion of 2-cell from $\mathbf{d}$ to $\mathbf{d}'$ (for details see~\cite{cs10}) and an obvious notion of 2-cell from $\varphi$ to $\varphi'$ (actually, there exists a 2-cell from $\varphi$ to $\varphi'$ if, and only if, for every $i\in I$, $\varphi(i)\leq \varphi'(i)$), we have 2-cells from  $(\varphi,\mathbf{d})$ to $(\varphi',\mathbf{d}')$, and, surely, the process described above would be 2-categorial.

\end{document}